\documentclass[11pt]{amsart}

\usepackage[a4paper,  left=25mm, right=25mm, top=25mm, bottom=25mm]{geometry}

\usepackage[utf8]{inputenc}
\usepackage{stmaryrd}
\usepackage{amsmath}
\usepackage{amssymb}
\usepackage{amsfonts}
\usepackage{amsthm}
\usepackage{bbm}
\usepackage{mathrsfs}
\usepackage{graphicx}
\usepackage{enumerate}
\usepackage[dvipsnames]{xcolor}
\usepackage[colorlinks=true,linkcolor=NavyBlue,urlcolor=RoyalBlue,citecolor=PineGreen,bookmarks=true,bookmarksopen=true,bookmarksopenlevel=2,unicode=true,linktocpage]{hyperref}

\usepackage[]{todonotes}

\newtheorem{theorem}{Theorem}[section]
\newtheorem{lemma}[theorem]{Lemma}
\newtheorem{proposition}[theorem]{Proposition}

\newtheorem{corollary}[theorem]{Corollary}

\newtheorem{assumption}[theorem]{Assumption}
\newtheorem{remark}[theorem]{Remark}

\def\scF{{\mathscr F}^{\text{\raisebox{1.5pt}{$\scriptscriptstyle\rightarrow$}}}}
\def\rF{{\sf F}}
\def\cyc{{\mathsf c}} 
\def\pcyc{{\mathsf p}} 
\def\Cyc{{\sf C}} 
\def\C{{\mathbb C}} 
\def\CMS{{\mathscr C}} 
\def\CMSf{\CMS_{\leq \ul{n}}}

\def\Map{{\mathfrak M}} 
\def\S{{\mathfrak{S}}} 
\def\cy{{(\!(}} 
\def\yc{{)\!)}}

\def\bl{{\mathsf{b}}} 

\def\build#1_#2^#3{\mathrel{\mathop{\kern 0pt#1}\limits_{#2}^{#3}}}

\def\Cyc{{\sf C}}

\newcommand{\diag}{\mathsf{diag}}
\def\E{{\sf E}}
\newcommand{\rE}{\mathsf{E}}
\newcommand{\bE}{\mathbb{E}}

\def\epsilon{\varepsilon}

\newcommand{\rG}{\mathsf{G}}
\def\ge{\geqslant}
\def\geq{\geqslant}
\def\GL{{\mathrm{GL}}}

\def\hb{h^{\text{\scalebox{0.8}{\rotatebox{90}{$\bowtie$}}}}}
\def\Hcyc{{\mathscr H}}
\def\HF{{\mathscr{H\! F}}^{\text{\raisebox{1.5pt}{$\scriptscriptstyle\rightarrow$}}}}
\newcommand{\hol}{\mathsf{hol}} 
\def\set#1{{\llbracket #1 \rrbracket}} 
\def\int#1{{[\![#1]\!]}}
\def\le{\leqslant}
\def\leq{\leqslant}

\newcommand{\cP}{\mathcal{P}} 

\def\phi{\varphi}
\newcommand{\proj}[1]{{\sf\Pi}^{#1}}

\def\s{{\sf s}}

\newcommand{\SL}{\mathrm{SL}}

\def\Tr{\mathop{\rm Tr}}

\newcommand{\cU}{\mathscr{U}}
\def\ul{\underline}

\newcommand{\rV}{\mathsf{V}} 
\def\val{{\rm val}} 
 
\newcommand{\cW}{\mathcal{W}} 
\def\Z{{\mathbb Z}}
\def\carre{\raisebox{-1.5pt}{\rotatebox{45}{$\diamond$}}} 

\title{Trace identities for quiver representations}

\author{Adrien Kassel}
\address{Adrien Kassel -- CNRS -- UMPA, ENS de Lyon}
\email{adrien.kassel@ens-lyon.fr}

\author{Thierry L\'evy}
\address{Thierry L\'evy -- LPSM, Sorbonne Universit\'e, Paris}
\email{thierry.levy@sorbonne-universite.fr}

\date{\today}



\begin{document}

\begin{abstract}
We give an expression for the determinant of the twisted Laplacian associated with any linear representation of a finite quiver in terms of traces of the holonomy of its cycles. To establish this expression, we prove a general identity for the determinant of a block matrix in terms of traces of products of its blocks. We give two proofs, one purely enumerative and one using generating series. 
       
In the special case of a finite graph equipped with a vector bundle and a connection, the twisted Laplacian determinant admits a combinatorial interpretation as a weighted count of tuples of oriented cycle-rooted spanning forests, where the weights involve traces of  holonomies along cycles formed by combining the edges of the forests.
\end{abstract}

\maketitle

\setcounter{tocdepth}{1}
{\small \tableofcontents}

\section{Introduction}

\subsection{Quivers, their linear representations, and the twisted Laplacian determinant}

\subsubsection{Quiver representations}\label{sec:graphsandlaplacian}

Let $\rG=(\rV,\rE)$ be a quiver, that is a finite directed graph, where~$\rV$ is the set of its vertices and~$\rE$ is the set of its (directed) edges. For each vertex~$v\in \rV$, let $n_v\ge 1$ be a positive integer, and let $\ul{U}=(U_e)_{e\in \rE}$ be complex matrices indexed by edges, where for each edge $e$ joining vertex $u$ to vertex $v$, the matrix~$U_e$ is of size $n_u\times n_v$. Since the matrices $U_e$ over edges have compatible dimensions when edges meet at a vertex, they can be multiplied along paths in the digraph: this is the setup of quiver representations, where edges of a directed graph are indexed by linear maps which can be composed along paths~\cite{Brion,Tyler}; see Figure~\ref{fig:quiver-rep}.

\begin{figure}[!ht]
    \centering
    \includegraphics[scale=.7]{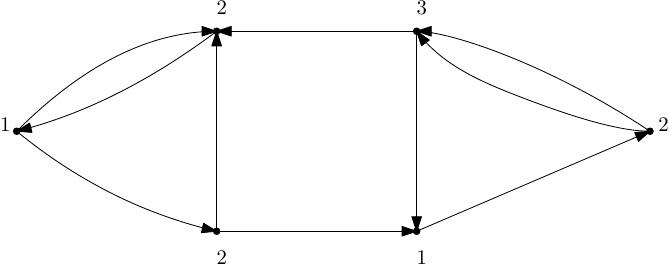}
    \caption{\small A quiver with integers on vertices representing the rank of the representation: it assigns to each arrow a matrix compatible with these dimensions (not shown).}
    \label{fig:quiver-rep}
\end{figure}

In the paper, it will be crucial to have a collection of edge weights $\ul{x}=(x_e)_{e\in \rE}$ attached to those directed edges. In general, we may think of these weights as complex numbers. Later in the paper when needed (see e.g. Section \ref{sec:euler-product}), we will specify that they be nonnegative real numbers. 

\subsubsection{Graphs versus quivers}\label{sec:graphvsquivers}

A special case of quivers is when all edges come with an inverse edge. Formally, this is the case, where there is a fixed-point free involution on $\rE$, denoted $e\mapsto e^{-1}$ such that $s(e)=t(e^{-1})$. 

In that case, we say that the quiver is a \emph{bidirected} graph, or simply a graph. In that case, we can visualize such a quiver by drawing a single bidirected arrow for each pair $\{e,e^{-1}\}$, see Figure \ref{fig:bidirected-graph}.

\begin{figure}[!ht]
    \centering
    \includegraphics[scale=.7]{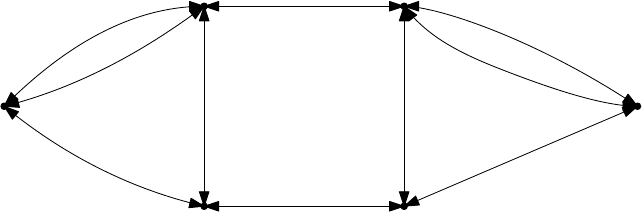}
    \caption{\small A bidirected graph which has a sub-quiver consisting in the quiver drawn in Fig.~\ref{fig:quiver-rep}.}
    \label{fig:bidirected-graph}
\end{figure}

When we do specify this additional structure, it is natural to impose some relations on the representation: if all $n_a$ are equal, we can ask the $U_e$ to be invertible and that $U_{e^{-1}}=U_e^{-1}$.

Another point of view is to always consider a quiver to be extended to a bidirected graph, without imposing any relation between the weights on $e$ and its inverse. Then we can always set $x_{e^{-1}}=0$ or $U_{e^{-1}}=0$ when we want effectively want the graph to really have only one of the two orientations for that edge which matters.

\subsubsection{Twisted quiver Laplacian}

An important operator associated to this data is the \emph{Laplacian} (also called \emph{covariant} Laplacian, \emph{connection} Laplacian, or simply, \emph{twisted} Laplacian), acting on $\prod_{v\in \rV} \C^{n_v}$ as follows: for each $f\in \prod_{v\in \rV} \C^{n_v}$ and each vertex~$v\in \rV$,  
\begin{equation}\label{eq:covLaplacian}
\Delta f (v) = \sum_{\substack{e\in \rE\\ s(e)=v}} x_e \left[ f\big(s(e)\big) - U_e f\big(t(e)\big) \right]\,,
\end{equation}
where $s(e)$ is the source-vertex, and $t(e)$ the tail-vertex, of the oriented edge $e$\footnote{Note that $U_e$ can be thought of as the matrix of a linear map from $\C^{n_v}$ to $\C^{n_u}$, hence, in the case of a bidirected graph, this is indeed the same twisted Laplacian matrix as the one defined in \cite{Kenyon, KL2, KL7} for instance.}. In any basis of the vector space~$\prod_{v\in \rV}\C^{n_v}$ adapted to its natural splitting indexed by vertices, the matrix of $\Delta$ is a $\vert \rV\vert\times \vert \rV\vert$ block matrix.

In the following, we will sometimes denote by $\C_a^{n_a}$ the vector space sitting over vertex $a$, so that $\Delta$ is an endomorphism of $\bigoplus_{a\in \rV}\C_a^{n_a}$.

\subsubsection{Kernel of the Laplacian}

If the quiver is acyclic (see Figure \ref{fig:acyclic-quiver}), then one may construct non-zero elements of the kernel of $\Delta$. Indeed, first observe that a finite acyclic quiver necessarily has a vertex, denote it $a$, with no outgoing edge. If this vertex has no incoming edge, it is isolated, and thus any vector in $\C_a^{n_a}$ extended to zero on other vertices is in the kernel of $\Delta$. If~$a$ has indegree at least~$1$, pick a non-zero vector $f_a$ in $\C_a^{n_a}$: then for any vertex $b$ connected to~$a$ by an edge~$e$, set $f_b=U_e f_a$. Continue until all vertices which are connected to $a$ are reached. Set $f$ equal to~$0$ on the other vertices (if any). Then $\Delta f=0$. 

For a general quiver, whether the kernel is zero or not depends on the representation and is essentially a reflection of the fact that the holonomies along cycles (that is the product of matrices $U_e$ along edges $e$ of the cycle) have joint eigenvectors of eigenvalue $1$.

\subsubsection{Gauge-invariance}

A quiver representation has automorphisms, the group of which is called the gauge group. 
These automorphisms are encoded by the natural action of $\prod_{v\in \rV}\GL_{n_v}(\C)$ (the so-called \emph{gauge transformations}), defined for each $(j_v)_{v\in \rV}\in \prod_{v\in \rV}\GL_{n_v}(\C)$ and edge $e$, by 
\[(j\cdot U)_e= j_{s(e)} U_e j_{t(e)}^{-1}\in M_{n_u,n_v}(\C)\,.\]

We are interested in functionals of the quiver representation which are invariant under automorphisms. For reasons inherent to the problems being modelled (see examples in~§\ref{sec:intro-examples}), one often wants to preserve, in the computation of any ``gauge-invariant observable quantity'' (that is, a functional) of this graphical data, which breaks it down to a sum of simpler terms, the gauge-invariance of the summands.

\subsubsection{Laplacian determinant}

A fundamental such observable is the determinant of $\Delta$, or the determinant of any principal block-minor $\Delta$, which we can view as a polynomial function of the edge-weights $\ul{x}$ and as a gauge-invariant function of the matrices $\ul{U}$. As noted above, if the quiver is acyclic, any representation will give a Laplacian with zero determinant.

For that matter, we will take an ``invariant'' approach to the computation of the determinant of block matrices, see Proposition \ref{prop:main}, which takes a particularly nice form when the diagonal blocks are scalar, see Theorem \ref{thm:main}, and its specialization to Corollary \ref{coro:detDelta}. Combined with Zeilberger's approach to matrix algebra \cite{Zeilberger} as used in~\cite{KL2}, we also obtain a second combinatorial interpretation in Theorem \ref{thm:formanN}. In passing, we give in Theorem \ref{thm:euler-product} another formula for $\det \Delta$, of the Euler infinite-product type, which is reminiscent of another formula of Forman on discrete dynamical systems and their zeta functions \cite[Theorem 7.5]{Forman-zeta}.

\subsection{Main results (I): abstract block matrices}

Here is a quick overview of our main results and how they sequentially unfold.

\subsubsection{General block matrices}

The starting point is a formula for the determinant of an $n\times n$ matrix 
\begin{equation}\label{eq:Ablocks-intro}
A=\begin{pmatrix} A_{[11]} &  \cdots & A_{[1p]} \\
\vdots & & \vdots \\
A_{[p1]} & \cdots & A_{[pp]}\end{pmatrix}\, ,
\end{equation}
which is written in block form as $A=\big(A_{[a_i a_j]}\big)_{1\le a_i,a_j\le p}$ with blocks of dimension $n_{a_i}\times n_{a_j}$, for integers $n_i\ge 1$ such that $\sum_{a=1}^p n_a=n$.

\begin{proposition}[Proposition \ref{prop:main}]
The following equality holds:
\begin{equation}\label{eq:dettrblock-intro}
\det A=\frac{1}{n_{1}!\ldots n_{p}!}\sum_{\sigma \in \S_{n}} \epsilon(\sigma) \prod_{(i_{1} \cdots i_{r})\preccurlyeq \sigma} 
\Tr\left(A_{[\bl(i_{1})\bl(i_{2})]} \ldots A_{[\bl(i_{r})\bl(i_{1})]}\right)
\end{equation}
where the product is over cycles of $\sigma$, and $\bl:\set{n}\to\set{p}$ is the map which to an index associates the index of the block to which it belongs.
\end{proposition}

Recognizing in the right-hand side of \eqref{eq:dettrblock-intro} a special case of the tau-determinant introduced in~\cite[Eq. (1)]{KL2}, namely the trace-determinant, we may reformulate the previous proposition as follows.

\begin{corollary}[Corollary \ref{coro:main}]
The following equality holds:
\begin{equation}
  n_1!\ldots n_p! \, \det A= {\det}_{\Tr}(A^{\carre})\,,
\end{equation}
where $A^{\carre}$ is an $n\times n$ matrix which also has a block-structure such that block $(a_i,a_j)$ is the matrix with entries all equal to the same non-commutative coefficient $A_{[a_ia_j]}$. 
\end{corollary}

The interpretation of the above formula is that, in order to compute a determinant of a block matrix without breaking the gauge symmetry, it is enough to replace each block by a block of non commuting entries which are all equal (to the initial block itself), and then compute the trace-determinant (upon dividing by a global symmetry factor $n_1! \ldots n_p!$).

\subsubsection{Block matrices with scalar diagonal blocks}

We then specialize this result to the case where the diagonal blocks are scalar: assume that for each $a\in \set{p}$, the diagonal block $A_{[aa]}$ is of the form $z_{a}I_{n_{a}}$ for some scalar $z_{a}$. This allows to separate in the previous formula the contribution of fixed blocks and of the so-called cyclic walks, that is cycles on the complete graph on the vertex set $\set{p}$. 

Let us introduce some notations: in the following, we let $\ul{n}=(n_a)_{1\le a\le p}$, let $\CMSf$ be the collection of multisets of cyclic walks, and for such a multiset $\Cyc$, let $\Cyc$ be the product of factorials of multiplicities, and $\ul{v}(\Cyc)$ be the total amount of vertices visited by the cycles, counting multiplicity. For $\cyc\in \Cyc$, we also let $\val(c)$ be the valuation of the cycle, and $\Tr(A_\cyc)$ be the trace of the product of the blocks of $A$ multiplied along $\cyc$. Finally, we let $\ul{z}^{\ul{n}-\ul{v}(\Cyc)}=\prod_{a=1}^p z_a^{n_a-v_a(\Cyc)}$. With these notations introduced, the simplification of our formula reads as follows.

\begin{theorem}[Theorem \ref{thm:main}] 
Let $A$ be a matrix with entries in a field of characteristic $0$ and with the structure depicted in \eqref{eq:Ablocks-intro}.
Assume that for each $a\in \set{p}$, the diagonal block $A_{[aa]}$ is of the form $z_{a}I_{n_{a}}$ for some scalar $z_{a}$. Then, the following equality holds:
\begin{equation}
\det A= \sum_{\Cyc \in \CMSf} \frac{\ul{z}^{\ul{n}-\ul{v}(\Cyc)} }{\Cyc!} 
 \prod_{\cyc\in \Cyc} \frac{(-1)^{\vert c\vert-1}\Tr(A_\cyc)}{\val(\cyc)}\, ,
\label{eq:dettrblockdiag-intro}
\end{equation}
where the product over $\cyc \in \Cyc$ takes multiplicities into account.
\end{theorem}

In Section \ref{sec:alternative-proof}, we give another proof of the previous theorem. This alternative proof further leads to the following identity.

\begin{proposition}[Proposition \ref{prop:Euler}]
With the notation of Theorem \ref{thm:main}, the following identity holds, with $\cP$ denoting the set of prime cycles on the complete graph on $\set{p}$:
\[\det A= \ul{z}^{\ul{n}} \prod_{\pcyc\in \cP} \det\big(I-\ul{z}^{-\ul{v}(\pcyc)}(-A)_{\pcyc})\big)\]
as an equality of Laurent series in $z_{1}^{-1},\ldots,z_{p}^{-1}$.
\end{proposition}

\subsection{Main results (II): the case of the Laplacian of a quiver representation}

We then specialize the above results to the case of a Laplacian $\Delta$ of a quiver representation as defined in \eqref{eq:covLaplacian} (which specializes to a graph with vector bundle and unitary connection, as explained above). We assume the vertex set $\rV$ to be $\set{p}$. 

For simplicity, we assume that the graph $\rG$ has no self-loops, that is edges $e$ such that~$s(e)=t(e)$.
In that case, the matrix of $\Delta$ in any basis adapted to the block decomposition $\C^n=\bigoplus_{a=1}^p \C^{n_p}$ has scalar diagonal blocks which we will still denote $z_a$ and which are related to the weight $x_{e}$ by
\begin{equation}\label{eq:defzu}
    z_a=\sum_{\substack{e\in \rE:\\s(e)=a}}x_{e}\,.
\end{equation}
Given a cycle $\cyc$ of $\rG$, we will denote $\hol(\cyc)$ the holonomy along a cycle $\cyc$ with respect to the quiver representation~$\ul{U}$, and $\ul{x}^{\ul{e}(\cyc)}$ the product of edges weights $x_e$ along edges $e$ of $\cyc$ (with multiplicity).

Furthermore, we let $\Cyc(\rG)_{\le \ul{n}}$ be the set of multisets of cycles on $\rG$ such that the total number of visits at each vertex $a$, does not exceed $n_a$.

\subsubsection{Wilson loop expansion}

\begin{theorem}[Corollary \ref{coro:detDelta}]
We have 
\begin{equation*}
\det \Delta= \sum_{\Cyc \in\CMS(\rG)_{\leq \ul{n}}} \frac{\ul{z}^{\ul{n}-\ul{v}(\Cyc)} }{\Cyc!} 
 \prod_{\cyc\in \Cyc} \frac{-\ul{x}^{\ul{e}(\cyc)}\Tr \hol(\cyc)}{\val(\cyc)}\,.
\end{equation*}
\end{theorem}

A similar expression is given for the characteristic polynomial of $\Delta$ in Theorem \ref{thm:polcarDelta}. Corollary~\ref{coro:detDelta} gives a practical way to compute expectations, or moments, of $\det \Delta$ when, for fixed~$\ul{n}$ we are given a distribution on the linear representations of $\rG$: this happens in the context of lattice gauge theory for instance where the representation plays the role of a gauge field, and where $\det \Delta$ can be the partition function for a model of matter (either bosonic \cite{KL1} or fermionic~\cite{KL7}). In that case, the linearity of expectation lets us rewrite moments of the Laplacian determinant as a linear combination (with explicit combinatorial weight) of expectations of products of traces of random holonomies along cycles, called Wilson loops (see Theorem~\ref{thm:gauge-theory}). These latter variables are often amenable to calculations, say in two-dimensional Yang--Mills theory~\cite{Levy1, Levy2}.

\begin{corollary}[Theorem \ref{thm:gauge-theory}]
Consider any probability distribution on the quiver representation for fixed $\ul{n}$. Then, for any integer $k\ge 1$, we have
\begin{equation*}
    \bE \left[\left(\det\Delta\right)^k \right] = 
    \sum_{\Cyc_1 , \ldots, \Cyc_k \in \CMS(\rG)_{\le \ul{n}}}
    \underbrace{\left(\prod_{i=1}^k \bigg[\frac{\ul{z}^{\ul{n}-\ul{v}(\Cyc_i)} }{\Cyc_i!} 
 \prod_{\cyc_i\in \Cyc_i} \frac{-\ul{x}^{\ul{e}(\cyc_i)}}{\val(\cyc_i)}\bigg] \right)}_{\text{Combinatorial weight}} \underbrace{\bE \left[\prod_{i=1}^k \prod_{\cyc_i\in \Cyc_i}\Tr \hol(\cyc_i) \right]}_{\text{Wilson loop}}\,.
\end{equation*}
\end{corollary}

\subsubsection{Discrete vector fields}

Expanding the above monomial in $\ul{z}$ using \eqref{eq:defzu}, we obtain the following expression in terms of vector fields on the graph, which is a generalization of \cite[Theorem 1]{Forman}.

\begin{theorem}[Theorem \ref{thm:forman-zeilberger-vector-fields}]\label{thm:formanN-intro}
We have 
\begin{equation}
   \det \Delta= \frac{1}{{n_{1}!\cdots n_p !}}\sum_{\xi\in \Xi} \; \ul{x}^{\xi} \; \bigg[\sum_{\sigma \in \Sigma(\xi)} \prod_{a=1}^{p}  (n_{a}-v_{a}(\sigma))!
\prod_{\cyc \in C(\sigma)} -\Tr \hol(\cyc)\bigg]\,.
\end{equation}
\end{theorem}

\subsubsection{Prime cycles and Euler products}

Assume that $x_{e}\geq 0$ for each edge $e$. For each vertex $a$, let $\kappa_a \ge 0$ and set for each edge $e\in \rE$
\begin{equation}
    p_e=\frac{x_e}{z_{s(e)}+\kappa_{s(e)}}\,.
\end{equation}
Let $\cP(\rG)$ denotes the set of prime cycles on $\rG$, that is cycles of valuation $1$.

\begin{theorem}[Theorem \ref{thm:euler-product}]
Under the assumption that the chain $p$ on $\rG$ is sub-Markovian, and a uniform bound on the matrix norm of the holonomies of cycles of $\rG$, we have 
\begin{equation}\label{eq:twisted-zeta-intro}
    \det \big(\diag(\kappa_a:a\in\rV)+\Delta\big)= (\ul{z}+\ul{\kappa})^{\ul{n}} \prod_{\cyc\in \cP(\rG)} \det\big(I-\ul{p}^{-\ul{e}(\cyc)}\hol(\cyc))\big)\,.
\end{equation}
\end{theorem}

Note that when the quiver has only a finite number of prime cycles, we can remove the assumption of sub-Markovianity. Indeed we can set $\kappa_a=0$ for all $a$, since we can exchange the limit and (finite) product in the right-hand side. 

\begin{corollary}[Theorem \ref{thm:finite-euler}]
If the quiver $\rG$ only has a finite number of prime cycles, we have 
\begin{equation}\label{eq:twisted-zeta-finite}
    \det \Delta= \ul{z}^{\ul{n}} \prod_{\cyc\in \cP(\rG)} \det\big(I-\ul{p}^{-\ul{e}(\cyc)}\hol(\cyc))\big)\,,
\end{equation}
where $p_e=x_e/z_{s(e)}$ for all $e\in \rE$.
\end{corollary}

Interestingly, the product in the right-hand side of \eqref{eq:twisted-zeta-intro} is reminiscent of the twisted Selberg's zeta function (see \cite[Eq. (4.4)]{Ray-Singer}). The fact that there is one infinite product less is due to the fact that we do not consider geodesic prime cycles, but all prime cycles (a geodesic cycle would be a non-backtracking cycle in our setup); the Ihara zeta function is the one appearing for those cycles.

\subsection{Background}

\subsubsection{Occurrence of the framework}\label{sec:intro-examples}

Apart from the theory of quiver representations per se, such a graphical setup arises in different contexts, notably in discrete differential geometry or lattice gauge theory, where the edge-weights are related to a distance or an interaction strength. When there is an integer $N$ such that~$n_v=N$ for all $v\in V$, the matrices model parallel transport, or a gauge interaction field on a rank $N$ vector bundle; see for instance~\cite{KL1,KL7} and references therein. These decorated graphs also appear in theoretical statistics, for instance in group synchronisation problems~\cite{Bandeira-Singer-Spielman}, or in vector diffusion maps for manifold learning and dimensional reduction~\cite{Singer-Wu}; see also~\cite{Gao-Brodzki-Mukherjee}. The case of non-square matrices (that is, when not all $n_v$ are equal) has also been used in applications, see for example \cite{Seigal} and \cite{Sumray-Harrington-Nanda}. 
It also occurs in a theory of sheaf on graphs introduced in~\cite{Friedman}.

\subsubsection{Computing determinants of block matrices}

Computing the determinant of a block matrix in terms of its blocks, without mixing their entries directly, is a natural problem which is likely to have been addressed many times in the literature. In particular, methods are known to compute this determinant by using repeatedly the Schur complement formula or variants of it. In this paper, we take a different approach to this computation, more in the spirit of `symbolic dynamics', and closely related to works on (weighted) zeta functions of graphs by Ihara, Bass, Foata, Zeilberger, Sato, and others, using cycles on graphs as combinatorial words.

\subsubsection{Special case}

Assume the quiver is a bidirected graph in the sense defined in Subsection~\ref{sec:graphvsquivers}. 
For a choice of symmetric edge-weights (that is, $x_e=x_{e^{-1}}$ for all $e\in \rE$), constant $n_v=N$, and unitary and orientation reversing matrices (that is, $U_e\in U_{n}(\C)$ and $U_{e^{-1}}=U_e^{-1}$ for all $e\in \rE$), certain powers of this determinant are shown to be the partition function of some probability distributions on functions over the vertices~\cite{KL1}, and Grassmannian-valued functions over the edges~\cite{KL3,KL7}. 

The computation of $\det \Delta$ in the case where $n_v=N$ for all $v\in V$ has been studied already. When $N=1$, a formula for $\det \Delta$, providing an interpretation of this determinant as a sum over discrete vector fields of the holonomy of limit-cycles of the induced finite dynamical system, was obtained by Forman, generalizing earlier work of Chaiken. This formula was extended by Kenyon to the case of $N=2$ with matrices in $\SL_2(\C)$; see \cite{KL2} and references therein. A natural question asked in \cite[Section 14, Open question 2]{Kenyon} is to find a combinatorial interpretation of this determinant when~$N\ge 2$ in general. Such combinatorial interpretations have been proposed in \cite{KL2,KL7}. The goal of this note is to propose yet another one, see Theorem~\ref{thm:formanN-intro}, which is simpler in a sense, and closer to the formula of Forman \cite{Forman}, and which moreover holds for general collections of integers~$\{n_v\ge 1, v\in \rV\}$.

\subsection{Organization of the paper}

In Section \ref{sec:block} we prove our main result relating the computation of a determinant of a block matrix to the computation of traces of products of its blocks. 
In Section \ref{sec:main} we specialize this formula to the case of scalar diagonal blocks.
In Section \ref{sec:alternative-proof} we provide a different proof of the result of the previous section. 
In Section \ref{sec:wilson-loops} we consider the special case of the twisted Laplacian of a finite quiver representation, thus obtaining a Wilson loop expansion for it.
In Section \ref{sec:vector-fields} we show how to rewrite our formulas in terms of ordered tuples of nonzero discrete vector fields on the quiver. 
In Section \ref{sec:euler-product} we give an Euler product expression for the twisted Laplacian.
In Section \ref{sec:examples} we compute a few explicit examples.
Section \ref{sec:conclusion} collects a few concluding remarks and questions for future research.
In Appendix \ref{sec:erratum}, we provide some clarification to \cite{KL2}.

\section{A formula for the determinant of a block matrix}\label{sec:block}

For each integer $m\ge 1$, we will use the notation $\set{m}=\{1,\ldots, m\}$, and we set $\set{0}=\varnothing$.

\subsection{A classical formula}

To start off, we are going to recall an expression of the determinant of a matrix as a polynomial function of the traces of the powers of this matrix. The formula that we give is very classical, and is an instance of the relations between Newton sums and elementary symmetric functions. The proof that we give, however, is not based on the use of symmetric functions, and we will then show how it extends to the case of block matrices.

Let us state the formula. We denote by $\S_{n}$ the group of permutations of $\set{n}$. If $\sigma$ is a permutation and $c=(i_{1}\, \cdots \, i_{r})$ is a cycle, we write $c\preccurlyeq \sigma$ to indicate that $c$ is a cycle of $\sigma$. We also denote by $|c|$ the length of $c$, which in this case is $r$. We denote by $\epsilon(\sigma)$ the signature of $\sigma$.

\begin{proposition}\label{prop:dettr1} Let $n\geq 1$ be an integer and let $A\in M_{n}(\C)$ be an $n\times n$ matrix with complex entries. The following equality holds:
\begin{equation}
\label{eq:dettr}
\det A=\frac{1}{n!}\sum_{\sigma\in \S_{n}} \epsilon(\sigma) \prod_{c \preccurlyeq \sigma} \Tr\big(A^{|c|}\big)\, .
\end{equation}
\end{proposition}

We already said that \eqref{eq:dettr} is a classical expression. To be explicit, it is a special case of the Frobenius relations between Schur functions and power sums (see \cite{MacDonald}): if $\lambda$ denotes the partition $(1,\ldots,1)$ of $n$, then the associated Schur function $s_{\lambda}$ is the elementary symmetric function~$e_{n}$, and the Frobenius relation, written as a sum over permutations rather than partitions, reads 
\[e_{n}=\frac{1}{n!}\sum_{\sigma\in \S_{n}} \epsilon(\sigma) p_{\sigma}\,,\] 
where $p_{\sigma}$ is the power sum (or Newton function) indexed by the partition corresponding to the cycle structure of the permutation $\sigma$.

Equation \eqref{eq:dettr} can also be obtained in another way, closer to some computations that we will do later in this note (in Section \ref{sec:alternative-proof}). This other way starts by expanding 
\begin{equation}\label{eq:exptrlog1}
\det(I_{n}+tA)=\exp \Tr \log(I_{n}+tA)
\end{equation}
in power series of $t$. For each integer $k\geq 0$, the coefficient of $t^{k}$ is a sum over $k$ integer indices that, to the price of combinatorial factors, can be ordered and interpreted as the lengths of the cycles of a permutation, to give
\[\det(I_{n}+tA)=\sum_{k\geq 0} \frac{t^{k}}{k!} \sum_{\sigma \in \S_{k}} \epsilon(\sigma) \prod_{c\preccurlyeq \sigma} \Tr(A^{|c|}).\]
The terms for $k> n$ vanish on the right-hand side, and the equality of the terms of degree $n$ yields our formula.

Let us now present another approach to the proof of Proposition~\ref{prop:dettr1}, which is of independent interest, and which we will then generalize to block matrices in Proposition~\ref{prop:main}. 

\begin{proof}[Proof of Proposition~\ref{prop:dettr1}] For every permutation $\alpha\in \S_{n}$, let us define the matrix $A_{\alpha}$ by setting
\begin{equation}\label{eq:defAa}
\forall i,j\in \set{n}, \ (A_{\alpha})_{ij}=A_{\alpha(i)\alpha(j)}\, .
\end{equation}
The matrix $A_{\alpha}$ is simply the matrix $PAP^{-1}$, where $P$ is the permutation matrix corresponding to $\alpha$ :
\[\forall i,j\in \set{n}, \ P_{ij}=\delta_{j,\alpha(i)}\, .\] 
In particular, $A_{\alpha}$ has the same determinant as $A$. 

Now, the definition \eqref{eq:defAa} makes sense for any map $\alpha:\set{n}\to \set{n}$, bijective or not. However, if~$\alpha$ is not bijective, it is not injective, and the matrix $A_{\alpha}$ has two identical columns (and indeed two identical rows), so that its determinant is $0$. Let us denote by $\Map_{n}$ the set of all maps from~$\set{n}$ to $\set{n}$. It results from these observations that
\[\det A=\frac{1}{n!}\sum_{\alpha \in \Map_{n}} \det A_{\alpha}.\]

Let us expand the determinant of $A_{\alpha}$ in the most usual way and push the sum over $\alpha$ as far to the right as we can. We find
\begin{align*}
\det A&=\frac{1}{n!}\sum_{\alpha \in \Map_{n}} \sum_{\sigma\in \S_{n}} \epsilon(\sigma) A_{\alpha(1)\alpha(\sigma(1))}\ldots A_{\alpha(n)\alpha(\sigma(n))}\\
&=\frac{1}{n!}\sum_{\sigma\in \S_{n}} \epsilon(\sigma)  \prod_{(i_{1}\, \cdots \, i_{r}) \preccurlyeq \sigma} \; \sum_{\alpha:\{i_{1},\ldots,i_{r}\}\to \set{n}} A_{\alpha(i_{1})\alpha(i_{2})}\ldots A_{\alpha(i_{r})\alpha(i_{1})}.
\end{align*}
Renaming, in the last sum, $\alpha(i_{1}),\ldots,\alpha(i_{r})$, which are now independent indices between $1$ and~$n$, as $j_{1},\ldots,j_{r}$, we find
\begin{align*}
\det A&=\frac{1}{n!}\sum_{\sigma\in \S_{n}} \epsilon(\sigma)  \prod_{(i_{1}\, \cdots\,  i_{r}) \preccurlyeq \sigma} \; \sum_{j_{1},\ldots,j_{r}=1}^{n} A_{j_{1}j_{2}}\ldots A_{j_{r}j_{1}}\\
&=\frac{1}{n!}\sum_{\sigma\in \S_{n}} \epsilon(\sigma)  \prod_{(i_{1}\, \cdots\,  i_{r}) \preccurlyeq \sigma} \Tr(A^{r}),
\end{align*}
which is the expected formula.
\end{proof}

Note that the argument shows that the equality \eqref{eq:dettr}, with both sides multiplied by $n!$, is valid for matrices with entries in an arbitrary commutative ring.

\subsection{A formula for block matrices} \label{sec:formblockmat}

We will now establish a formula similar to \eqref{eq:dettr} for a square matrix~$A$ that is given as a block matrix:
\begin{equation}\label{eq:Ablocks}
A=\begin{pmatrix} A_{[11]} &  \cdots & A_{[1p]} \\
\vdots & & \vdots \\
A_{[p1]} & \cdots & A_{[pp]}\end{pmatrix}\,,
\end{equation}
all diagonal blocks being square matrices. Our goal is to express the determinant of such a matrix as a linear combination of products of traces of the form
\[\Tr(A_{[a_{1}a_{2}]}\ldots A_{[a_{r}a_{1}]}).\]
Let us emphasize that \eqref{eq:dettr} can be turned into a formula of this kind simply by expanding blockwise the traces of powers of $A$. However, the formula that we will prove is not the formula that one would obtain in this way, and it is simpler. 

Let us define our notation more precisely. Let $p\geq 1$ be an integer, the size of $A$ as a block matrix. Let $n_{1},\ldots,n_{p}\geq 1$ be the sizes of the blocks, and $n=n_{1}+\cdots +n_{p}$ the total size of $A$. We denote the blocks of $A$ with indices between square brackets,
so that for all $a,b\in \set{p}$, the block $A_{[ab]}$ belongs to $M_{n_{a}n_{b}}(\C)$.
 
We keep the notation $A_{ij}$ for the scalar entry of $A$ located on the $i$-th row and the $j$-th column. 

Let us split the interval $\set{n}$ according to the block structure of $A$ by setting
\[I_{1}=\{1,\ldots,n_{1}\}, \ I_{2}=\{n_{1}+1,\ldots,n_{1}+n_{2}\}, \ \ldots, I_{p}=\{n_{1}+\cdots+n_{p-1}+1,\ldots n\}.\]
Let us also define, for each $i\in \set{n}$, the index $\bl(i)$ of the interval (or block) containing $i$, that is, the unique element of $\set{p}$ such that $i\in I_{\bl(i)}$. Thus, for all $i,j\in \set{n}$, the entry $A_{ij}$ of $A$ is located in the block $A_{[\bl(i)\bl(j)]}$.

\begin{proposition}\label{prop:main}
With the current notation, the following equality holds:
\begin{equation}\label{eq:dettrblock}
\det A=\frac{1}{n_{1}!\ldots n_{p}!}\sum_{\sigma \in \S_{n}} \epsilon(\sigma) \prod_{(i_{1} \cdots i_{r})\preccurlyeq \sigma} 
\Tr\left(A_{[\bl(i_{1})\bl(i_{2})]} \ldots A_{[\bl(i_{r})\bl(i_{1})]}\right)\,.
\end{equation}
\end{proposition}

For example, if $p=1$, then the trace that appears in the right-hand side is $\Tr(A^{r})$ and we recover \eqref{eq:dettr}. In the other extreme example where $p=n$ and all the blocks have size $1$, the function~$\bl$ is the identity of $\set{n}$, and \eqref{eq:dettrblock} is simply the usual formula for the determinant.

\begin{proof} We will do the same computation as in the proof of Proposition \ref{prop:dettr1}, with the only difference that we will now restrict ourselves to maps $\alpha\in \Map_{n}$ which preserve the block structure of $A$, in the sense that $\bl\circ \alpha=\bl$. Concretely, these are the maps which apply each interval $I_{1},\ldots,I_{p}$ into itself. The number of such maps which are moreover bijective is $n_{1}! \ldots n_{p}!$, so that the first part of the argument now reads
\[\det A=\frac{1}{n_{1}!\ldots n_{p}!} \sum_{\substack{\alpha \in \Map_{n}\\ \bl\circ \alpha=\bl}} \det A_{\alpha}.\]
Expanding the determinant of $A_{\alpha}$, we find
\[\det A=\frac{1}{n_{1}!\ldots n_{p}!} \sum_{\sigma\in \S_{n}} \epsilon(\sigma) \prod_{(i_{1}\, \cdots \, i_{r})\preccurlyeq \sigma}  
\; \sum_{\substack{\alpha:\{i_{1},\ldots,i_{r}\}\to \set{n}\\ 
{\bl}\circ \alpha={\bl}}} A_{\alpha(i_{1})\alpha(i_{2})}\ldots A_{\alpha(i_{r})\alpha(i_{1})}.\]
We now rename the indices $\alpha(i_{1}),\ldots,\alpha(i_{r})$ as $j_{1},\ldots,j_{r}$, which are still independent of each other. The difference with the case of Proposition \ref{prop:dettr1} is that $j_{1}$ is constrained to belong to the same block as $i_{1}$, and so on. Thus, we have
\[\det A=\frac{1}{n_{1}!\ldots n_{p}!} \sum_{\sigma\in \S_{n}} \epsilon(\sigma) \prod_{(i_{1}\, \cdots\,  i_{r})\preccurlyeq \sigma} \; \sum_{j_{1}\in I_{\bl(i_{1})}, \ldots, j_{r}\in I_{\bl(i_{r})}} A_{j_{1}j_{2}}\ldots A_{j_{r}j_{1}}\]
from which the announced formula follows immediately.
\end{proof}

As in the previous section, this formula, with both sides multiplied by $n_{1}!\ldots n_{p}!$, holds for matrices $A$ with entries in an arbitrary commutative ring.

\subsection{A formula in terms of matrices with non-commuting entries}

We will now rewrite the statement of Proposition \ref{prop:main} using the notion of trace-determinant, a special case of the notion of $\tau$-determinant that we introduced in \cite{KL2}.\footnote{The content of this subsection is not needed for the rest of the paper, but it is interesting in its own right, and may be used in future work.}

The minimal setting is the following. We are given a multiplicative monoid $R$, for instance a ring, commutative of not, with or without unit. We are also given a commutative ring $S$, and a map $\tau : R\to S$ that is central in the sense that $\tau(uv)=\tau(vu)$ for all $u,v\in R$.
Then, to a matrix $M\in M_{n}(R)$ we associate the $\tau$-determinant
\[{\det}_{\tau}(M)= \sum_{\sigma \in \S_n} \varepsilon(\sigma) \prod_{(i_{1}\, \cdots\,  i_{r}) \preccurlyeq \sigma} \tau(M_{i_{1}i_{2}}\ldots M_{i_{r}i_{1}}),\]
which is an element of $S$.

For our purposes, we define the ring
\[R=\bigoplus_{k,l\geq 1} M_{k,l}(\C),\]
in which the product of two matrices of sizes $(k,l)$ and $(k',l')$ is defined to be zero if $l\neq k'$. We take $S=\C$, and define $\tau$ to be the usual trace on $M_{k,l}(\C)$ when $k=l$, and zero when $k\neq l$. With these definitions, for any two matrices $U$ and $V$, either $U$ and $V$ have transposed sizes, in which case $UV$ and $VU$ are square and $\tau(UV)=\tau(VU)$; or they do not, in which case, at least one of the products $UV$ and $VU$ is not a square matrix, and the other is equal to $0$, so that $\tau(UV)=\tau(VU)=0$. In this setting, we use the notation ${\det}_{\Tr}$ for the $\tau$-determinant. 

Let us come back to the situation of the block matrix $A$ studied in the previous section. To~$A$, let us associate the matrix $A^{\carre}\in M_{n}(R)$ defined by setting, for all $i,j\in \set{n}$,
\[A^{\carre}_{ij}=A_{[\bl(i)\bl(j)]}.\]
In English, we replace each entry of $A$ by the block which contains it. To be clear, the matrix~$A^{\carre}$ has the same size as the matrix $A$, but its entries are not scalars anymore: they are matrices. Moreover, these entries are constant on each block $I_{a}\times I_{b}$.

With this notation, Proposition \ref{prop:main} can be formulated as follows.

\begin{corollary}\label{coro:main}
The following equality holds:
\begin{equation}
  n_1!\ldots n_p! \, \det A= {\det}_{\Tr}(A^{\carre})\,.
\end{equation}
\end{corollary}

\section{Block matrices with scalar diagonal blocks}\label{sec:main}

In this section, we come back to the setting of Section \ref{sec:formblockmat} and we investigate the extent to which \eqref{eq:dettrblock} can be simplified when it is applied to a block matrix in which the diagonal blocks are scalar. To be precise, we will from now on make the assumption that for each $a\in \set{p}$, there exists a scalar $z_{a}$ such that the diagonal block $A_{[aa]}$ is the matrix $z_{a}I_{n_{a}}$. 

Our reformulation of \eqref{eq:dettrblock} under this new assumption involves multisets of cyclic walks on~$\set{p}$, which we define and for which we introduce some notation. 

\subsection{Cyclic walks} \label{sec:cyclicwalks}

First of all, by a {\em cyclic walk} on $\set{p}$, we mean a cyclically ordered list of length at least $2$ of elements of $\set{p}$, in which any two successive elements are distinct. We denote a cyclic walk of length $k\geq 2$ by $\cy a_{1},\ldots,a_{k}\yc=\cy a_{i},\ldots,a_{k},a_{1},\ldots,a_{i-1}\yc$, using double brackets to distinguish them from the cycles of permutations, and to be clear, we assume that $a_{1}\neq a_{2},\ldots, a_{k-1}\neq a_{k}, a_{k}\neq a_{1}$. The length of a cyclic walk $\cyc$ is denoted by $|\cyc|$.

Note that a cyclic walk may contain several times the same element of $\set{p}$. For example, $\cy1,2,1,2\yc$ and $\cy 1,2,1,3 \yc$ are cyclic walks, whereas $\cy 1,2,3,1\yc$ is not. 

A cyclic walk $\cyc=\cy a_{1},\ldots,a_{k}\yc$ of length $|\cyc|=k$ can be understood as the orbit of the $k$-tuple $(a_{1},\ldots,a_{k})$ of elements of $\set{p}$ under the action by cyclic permutations of $\Z/k\Z$. Each element of this orbit has the same stabiliser under this action, and the order (that is, the cardinality) of this stabiliser is called the {\em valuation} of the cyclic walk. It is denoted by\begin{equation}\label{eq:valuation}
\val(\cyc)=\big\vert \mathrm{Stab}_{ \Z/\vert \cyc\vert\Z}(\cyc)\big\vert\,.
\end{equation}
For example, the valuation of $\cy1,2,1,2\yc$ is $2$ and the valuation of $\cy 1,2,1,3 \yc$ is $1$. 

The number of visits of a cyclic walk $\cyc=(\!(a_{1},\ldots,a_{k})\!)$ at a vertex $a\in \set{p}$ is defined as the number
\[v_{a}(\cyc)=|\{i\in \set{k} : a_{i}=a\}|.\]
We will use the notation 
\[\ul{v}(\cyc)=(v_{1}(\cyc),\ldots,v_{p}(\cyc))\]
for the vector of these numbers of visits at all the elements of $\set{p}$.

Given a set, or multiset, $\Cyc$ of cyclic walks, we denote for each $a\in \set{p}$ by $v_{a}(\Cyc)$ the total number of visits at $a$ of the elements of $\Cyc$:
\begin{equation}\label{eq:visits}
v_{a}(\Cyc)=\sum_{\cyc\in \Cyc} v_{a}(\cyc)\,,
\end{equation}
where the sum takes into account the multiplicities of the elements of $\Cyc$. We also define the vector of integers $\ul{v}(\Cyc)=(v_{1}(\Cyc),\ldots,v_{p}(\Cyc))$.

Let us denote by $\CMS$ the set of multisets of cyclic walks on $\set{p}$ and by $\CMSf$ the subset of $\CMS$ formed by multisets of cyclic walks for which the total number of visits at each vertex $a\in \set{p}$ is not greater than $n_{a}$:
\begin{equation}\label{eq:def-CMS}
\CMSf=\{\Cyc \text{ multiset of cyclic walks } : \forall a\in \set{p}, \,  v_{a}(\Cyc)\leq n_{a}\}.
\end{equation}

For each such multiset $\Cyc$, we denote by $\Cyc!$ the product of the factorials of the multiplicities of the elements of $\Cyc$. Note that the number $\Cyc!$ is equal to $1$ if and only if every element of $\Cyc$ has multiplicity $1$, that is, if and only if $\Cyc$ is a set.

Finally, given a cyclic walk $\cyc=\cy a_{1},\ldots,a_{k}\yc$, the matrix $A_{[a_{1}a_{2}]}\ldots A_{[a_{k}a_{1}]}$ is ill-defined, as even its size depends on the chosen starting point $a_{1}$ of the cyclic walk. However, the trace of this matrix does not depend on this choice, and we define\footnote{The letter W is mnemonic for \emph{Wilson loop}, a quantity (the \emph{expectation} of the trace of the holonomy of a loop, with respect to a \emph{random} connection)  similar to $W(U,\cyc)$ which appears in lattice gauge theory, a field pioneered by K. Wilson. See also Theorem \ref{thm:gauge-theory}.} 
\begin{equation}\label{eq:defhAcyc}
W(A,\cyc)=\Tr(A_{[a_{1}a_{2}]}\ldots A_{[a_{k}a_{1}]}).
\end{equation}

We will use the notation $\ul{n}=(n_{1},\ldots,n_{p})$ for the vector of sizes of the blocks of $A$. For a vector of integers $\ul{m}=(m_{1},\ldots,m_{p})$, we will use the shorthand notation 
\[\ul{z}^{\ul{m}}=z_{1}^{m_{1}}\ldots z_{p}^{m_{p}}.\] 

We can now state the new version of \eqref{eq:dettrblock}.

\begin{theorem}\label{thm:main} 
Let $A$ be a matrix with entries in a field of characteristic $0$ and with the structure depicted in \eqref{eq:Ablocks}.
Assume that for each $a\in \set{p}$, the diagonal block $A_{[aa]}$ is of the form $z_{a}I_{n_{a}}$ for some scalar $z_{a}$. Then, with the notation introduced above, the following equality holds:
\begin{equation}
\det A= \sum_{\Cyc \in \CMSf} \frac{\ul{z}^{\ul{n}-\ul{v}(\Cyc)} }{\Cyc!} 
 \prod_{\cyc\in \Cyc} \frac{-W(-A,\cyc)}{\val(\cyc)}\, ,
\label{eq:dettrblockdiag}
\end{equation}
where the product over $\cyc \in \Cyc$ takes multiplicities into account.
\end{theorem}
 
Note that in the sum, the empty multiset $\Cyc$ needs to be taken into account and contributes with $\ul{z}^{\ul{n}}$.

Let us also indicate that after multiplying both sides by $n_{1}!\ldots n_{p}!$, the equality above involves only integer coefficients, and is valid for a matrix $A$ with entries in an arbitrary commutative ring, see Section \ref{sec:thmmainintegral}.

The case where $p=1$ of the theorem is rather trivial, since $A$ is diagonal. In the case where $p=n$ however, \eqref{eq:dettrblockdiag} is the usual formula for the determinant, with fixed points of permutations separated from non-trivial cycles.

Theorem \ref{thm:main} has the following corollary, which computes the characteristic polynomial of $A$. Let $T=\diag(t_{1}I_{n_{1}},\ldots,t_{p}I_{n_{p}})$ be an arbitrary block-diagonal matrix with scalar diagonal blocks. 

\begin{corollary}\label{coro:thm-main} 
\begin{equation*}
\det(T+A)=\sum_{\ul{0}\leq \ul{k}\leq \ul{n}} \ul{t}^{\ul{k}}\,  \sum_{\substack{\Cyc \in \CMS \\ \ul{v}(\Cyc)+\ul{k}\leq \ul{n}}}
\binom{n_{1}-v_{1}(\Cyc)}{k_{1}}\ldots \binom{n_{p}-v_{p}(\Cyc)}{k_{p}}
 \frac{\ul{z}^{\ul{n}-\ul{k}-\ul{v}(\Cyc)} }{\Cyc!} 
 \prod_{\cyc\in \Cyc} \frac{- W(-A,\cyc)}{\val(\cyc)}.
 \end{equation*}
\end{corollary}

\begin{proof} Simply replace $\ul{z}$ by $\ul{z}+\ul{t}$ and use the binomial identity for each $a\in \set{p}$.
\end{proof}

The proof of Theorem \ref{thm:main} consists in a computation that is not difficult, but involves several steps, that we present in great detail. It consists of the seven subsections \ref{sec:preparation} to \ref{sec:proofcounting}. 
Another proof of Theorem \ref{thm:main} is given in Section \ref{sec:alternative-proof}.

\subsection{Kinematic permutations} \label{sec:preparation}

Let us start by describing an operation which takes a permutation $\sigma$ of $\set{n}$ and produces a subset $R_{\sigma}$ of $\set{n}$ and a fixed point free permutation~$\delta_{\sigma}$ of this subset. This operation depends on the partition $\set{n}=I_{1}\sqcup \ldots \sqcup I_{p}$ corresponding to the block structure of the matrix $A$.  In words, $\delta_{\sigma}$ is obtained from $\sigma$ by removing from its cycle any element of $\set{n}$ that is in the same block as its predecessor. Any cycle of $\sigma$ that is contained in a single block is removed. The set~$R_{\sigma}$ is then just the support of $\delta_{\sigma}$, that is, the set of elements that have not been removed. It is possible that~$R_{\sigma}$ is empty: this happens exactly if $\sigma$ preserves the block structure. 

More precisely, we start by setting
\[R_{\sigma}=\{i\in \set{n} : \bl(\sigma^{-1}(i))\neq\bl(i)\}\,.\]
Then, for all $i\in R_{\sigma}$, we define
\[t(\sigma,i)=\min\{k\geq 1 : \sigma^{k}(i)\in R_{\sigma}\}=\min\{k\geq 1 : \bl(\sigma^{k}(i))\neq \bl(i)\}.\]
Note that since the cycle of any element $i$ of $R_{\sigma}$ contains at least two elements which do not belong to the same block, the set of which $t(\sigma,i)$ is the minimum is not empty. Finally, we define the permutation $\delta_{\sigma}$ of $R_{\sigma}$ by setting, for all $i\in R_{\sigma}$,
\[\delta_{\sigma}(i)=\sigma^{t(\sigma,i)}(i).\]
If $R_{\sigma}$ is empty, $\delta_{\sigma}$ is the unique map from $R_{\sigma}$ to itself.

For example, if $n=10$, $n_{1}=4$ and $n_{2}=n_{3}=3$, so that
\[\set{10}=\{1,\ldots,10\}=\underbrace{\{1,2,3,4\}}_{I_{1}} \sqcup \underbrace{\{5,6,7\}}_{I_{2}}\sqcup \underbrace{\{8,9,10\}}_{I_{3}},\]
and if $\sigma=(1\, 5 \, 6\, 3) (2\, 4) (8\, 7\, 9 \, 10)$, then $R_{\sigma}=\{3,5,7,9\}$ and $\delta_{\sigma}=(3\, 5)(7\, 9)$. 

The permutation $\delta_{\sigma}$ is {\em kinematic} in the sense that it changes blocks at each step: $\bl(\delta_{\sigma}(i))$ is never equal to $\bl(i)$. 

We will denote by $K(R)$ the set of kinematic permutations of a subset $R$ of $\set{n}$. The set $K(R)$ can be empty, and even when it is not, it is never a subgroup of the group of permutations of $R$. 

\subsection{Three more pieces of notation}

Firstly, given a cycle $c=(i_{1}\cdots i_{r})$ of a permutation of~$\set{n}$, or of a permutation of a subset of $\set{n}$, let us define
\begin{equation}\label{eq:defhAcycperm}
W(A,c)=\Tr(A_{[\bl(i_{1})\bl(i_{2})]}\ldots A_{[\bl(i_{r})\bl(i_{1})]}).
\end{equation}
With the notation \eqref{eq:defhAcyc}, this is $W(A,\cyc)$ where we have set $\cyc=\cy \bl(i_{1}),\ldots,\bl(i_{r}) \yc$.
\medskip

Secondly, for all subset $R$ of $\set{n}$ and all $a\in \set{p}$, we set
\begin{equation}
|R|_{a}=|R\cap I_{a}|.
\end{equation}
For a cyclic permutation $c$ of a subset of $\set{n}$, with support ${\rm Supp}(c)$, we set $|c|_{a}=|{\rm Supp}(c)|_{a}$.
\medskip

Thirdly, for all permutation $\delta$ of a subset $R$ of $\set{n}$, we denote by $\ell(\delta)$ the number of cycles of~$\delta$. Let us emphasize that we do not count the elements of the complement of $R$ in $\set{n}$ as fixed points of $\delta$. Thus, in the example shown a few paragraphs above, $\ell(\delta_{\sigma})=2$. For a permutation~$\sigma$ in $\S_{n}$, the number $\ell(\sigma)$ is just the usual number of cycles of $\sigma$.

Moreover, for all $\sigma\in \S_{n}$ and all $a\in \set{p}$, we set
\[\ell_{a}(\sigma)=\text{ number of cycles of } \sigma \text{ entirely contained in } I_{a}.\]
With these definitions, we have, for each permutation $\sigma$ of $\set{n}$, the relation
\begin{equation}\label{eq:ellell}
\ell(\sigma)=\ell(\delta_{\sigma})+\sum_{a=1}^{p} \ell_{a}(\sigma),
\end{equation}
the first term counting the cycles of $\sigma$ which visit at least two blocks, and the second term being the number of cycles of $\sigma$ which stay within one block.

\subsection{First step : resummation over kinetic permutations}

With all this preparation, we are now in position to give an alternative expression of \eqref{eq:dettrblock} under our assumption that the diagonal blocks of $A$ are scalar. The goal of this first step of the proof is to establish \eqref{eq:amusebouche}.

Let us consider $\sigma\in \S_{n}$, and a cycle $c=(i_{1}\cdots i_{r})$ of $\sigma$. We want to compute
\[ \Tr\left(A_{[\bl(i_{1})\bl(i_{2})]} \ldots A_{[\bl(i_{r})\bl(i_{1})]}\right).\]
If $c$ stays inside $I_{a}$ for some $a\in \set{p}$, then this trace is equal to 
\[n_{a} z_{a}^{|c|} = n_{a} z_{a}^{|c|_{a}}.\]
If not, then the permutation $\delta_{\sigma}$ has a unique cycle with support contained in that of $c$: let us denote it by $c'$. Then the trace is equal to 
\[W(A,c')\prod_{a=1}^{p} z_{a}^{|c|_{a}-|c'|_{a}}\]
Taking the product over all cycles of $\sigma$, and observing that for all $a\in \set{p}$, 
\[\sum_{c\preccurlyeq \sigma} |c|_{a}=n_{a} \ \text{ and } \ \sum_{c\preccurlyeq \sigma} |c'|_{a}=|R_{\sigma}|_{a},\]
we find 
\[\prod_{(i_{1} \cdots i_{r})\preccurlyeq \sigma} \Tr\left(A_{[\bl(i_{1})\bl(i_{2})]} \ldots A_{[\bl(i_{r})\bl(i_{1})]}\right)=\prod_{a=1}^{p} z_{a}^{n_{a}-|R_{\sigma}|_{a}} n_{a}^{\ell_{a}(\sigma)} \prod_{c\preccurlyeq \delta_{\sigma}} W(A,c)\,.\]

From \eqref{eq:ellell}, we deduce the following relation between the signatures of $\sigma$ and $\delta_{\sigma}$:
\[\epsilon(\sigma)=(-1)^{n+\ell(\sigma)}=(-1)^{|R_{\sigma}|+\ell(\delta_{\sigma})} (-1)^{n+|R_{\sigma}|}\prod_{a=1}^{p} (-1)^{\ell_{a}(\sigma)}=\epsilon(\delta_{\sigma}) (-1)^{n+|R_{\sigma}|}\prod_{a=1}^{p} (-1)^{\ell_{a}(\sigma)}.\]
The signature of $\delta_{\sigma}$ can be absorbed in the terms $W(A,c)$ by writing
\[\epsilon(\delta_{\sigma})\prod_{c\preccurlyeq \delta_{\sigma}} W(A,c)=\prod_{c\preccurlyeq \delta_{\sigma}} -W(-A,c)\,.\]
Combining the last three equations, \eqref{eq:dettrblock} can be rewritten as
\begin{equation}\label{eq:apero}
\det A=\sum_{\sigma\in \S_{n}}(-1)^{n+|R_{\sigma}|} \prod_{a=1}^{p} \frac{1}{n_{a}!}z_{a}^{n_{a}-|R_{\sigma}|_{a}} (-n_{a})^{\ell_{a}(\sigma)} \prod_{c\preccurlyeq \delta_{\sigma}} -W(-A,c)\,.
\end{equation}
Let us reorganize this sum by summing first over permutations which yield a given pair $(R,\delta)$:
\begin{equation}\label{eq:amusebouche}
\det A=\sum_{R\subseteq \set{n}} (-1)^{n+|R|} \prod_{a=1}^{p} \frac{z_{a}^{n_{a}-|R|_{a}}}{n_{a}!} \sum_{\delta\in K(R)}\bigg(\sum_{\substack{\sigma\in \S_{n} \\ \delta_{\sigma}=\delta}} \prod_{a=1}^{p}(-n_{a})^{\ell_{a}(\sigma)}\bigg) \prod_{c\preccurlyeq \delta} -W(-A,c).
\end{equation}

\subsection{Second step : contribution of each kinetic permutation}

Given $R$ and $\delta$, we determine in how many ways the elements of~$\set{n}\setminus R$ can be inserted in the cycles of $\delta$ in order to produce a permutation $\sigma$ such that~$R_{\sigma}=R$ and~$\delta_{\sigma}=\delta$, and we count the total contribution of these permutations.
\medskip

The goal of this second step is to establish \eqref{eq:interessante}. To do this, we will prove that 
\begin{equation}\label{eq:resto}
\sum_{\substack{\sigma\in \S_{n} \\ \delta_{\sigma}=\delta}} \prod_{a=1}^{p}(-n_{a})^{\ell_{a}(\sigma)}=(-1)^{n+|R|} \prod_{a=1}^{p} (n_{a}-|R|_{a})!
\end{equation}
To prove this equation, given $R$ and $\delta$, we will add the elements of $\set{n}\setminus R$ in the cycles of $\delta$ one by one. In order for $R$ and $\delta$ to stay unchanged during the process, each new element must be either inserted immediately after an element of the same block, or used to create a new cycle. This procedure is a variation on the theme of Pitman's so-called ``\emph{Chinese restaurant}'' algorithm for sampling a uniform permutation \cite{Pitman}.

Let us focus on the first block first and consider the smallest element of $I_{1}\setminus (R\cap I_{1})$. This element can be either inserted immediately after one of the $|R|_{1}$ elements of $\delta$ belonging to the first block, or used to create a new cycle. In the first case, the number of cycles of $\delta$ within $I_{1}$ is not altered, whereas in the second case, it is increased by $1$. Thus, the various possible ways of inserting this smallest element contributes a factor
\[|R|_{1} \times 1 +1 \times (-n_{1})=|R|_{1}-n_{1}.\]
The insertion of the second smallest element of $I_{1}\setminus (R\cap I_{1})$ contributes almost with the same factor, with the only modification that there are now $|R|_{1}+1$ elements of the first block in the permutation that we are building. Thus, we get a factor
\[(|R|_{1}+1)\times 1+ 1\times (-n_{1})=|R|_{1}+1-n_{1}.\]
We continue this process until the insertion of the greatest element of $I_{1}\setminus (R\cap I_{1})$, which contributes a factor
\[(n_{1}-1)-n_{1}=-1.\]
The total contribution of the first block is thus
\[(|R|_{1}-n_{1})(|R|_{1}+1-n_{1})\ldots (-1)=(-1)^{|R|_{1}+n_{1}} (n_{1}-|R|_{1})!\]
Each block produces a similar contribution, and \eqref{eq:resto} is proved. 

Combining \eqref{eq:amusebouche} and \eqref{eq:resto}, we find the following expression of the determinant of $A$:
\begin{equation}\label{eq:interessante}
\det A=\sum_{R\subseteq \set{n}} \prod_{a=1}^{p} \frac{(n_{a}-|R|_{a})!}{n_{a}!}z_{a}^{n_{a}-|R|_{a}} \sum_{\delta\in K(R)}\prod_{c\preccurlyeq \delta} -W(-A,c).
\end{equation}

\subsection{Third step : independence of the support of the kinetic permutation} \label{sec:indepsupport}

In this third step, we prove \eqref{eq:option2}, where the set $R_{\ul{r}}$ is defined in the middle of the present section. 

The key to this new step is the fact, that we will now prove, that the sum over $\delta$ in \eqref{eq:interessante} depends on $R$ only through the list of integers $(|R|_{1},\ldots,|R|_{p})$. 

Indeed, suppose $R$ and $R'$ are two subsets of $\set{n}$ such that for all~$a\in \set{p}$, we have $|R|_{a}=|R'|_{a}$. Then there exists a block-preserving bijection $\rho:R'\to R$, that is, a bijection such that $\bl\circ \rho=\bl$. Then the conjugation by $\rho$ is a bijection between $K(R)$ and $K(R')$, and for all $\delta'\in K(R')$ we have
\[\prod_{c\preccurlyeq \rho\delta'\rho^{-1}} -W(-A,c)=\prod_{c'\preccurlyeq  \delta'} -W(-A,\rho c'\rho^{-1})=\prod_{c'\preccurlyeq  \delta'} -W(-A,c').\]
Thus,
\[\sum_{\delta\in K(R)}\prod_{c\preccurlyeq \delta} -W(-A,c)=\sum_{\delta'\in K(R')}\prod_{c'\preccurlyeq \delta'} -W(-A,c').\]

Given a $p$-tuple of integers $\ul{r}=(r_{1},\ldots,r_{p})\in (\{0\}\sqcup\set{n_1})\times \ldots \times (\{0\}\sqcup\set{n_p})$, let us define
\[R_{\ul{r}}=\bigsqcup_{a=1}^{p}\{(n_{1}+\ldots +n_{a-1})+1,\ldots,(n_{1}+\ldots +n_{a-1})+r_{a}\},\]
as a reference subset of $\set{n}$ such that for each $a\in \set{p}$, we have $|R_{\ul{r}}|_{a}=r_{a}$. Then
\[\det A=\sum_{\substack{r_{1},\ldots,r_{p} \\ 0\leq r_{a} \leq n_{a}}} \prod_{a=1}^{p} \frac{z_{a}^{n_{a}-r_{a}}}{r_{a}!} \binom{n_{a}}{r_{a}}^{-1} \sum_{\substack{R\subseteq \set{n} \\ |R|_{a}=r_{a}}}  \sum_{\delta\in K(R_{\ul{r}})} \prod_{c\preccurlyeq \delta} -W(-A,c).\]
In this expression, nothing depends on $R$, and the corresponding sum is exactly compensated by the product of the binomial coefficients. Therefore,
\begin{equation}\label{eq:option2}
\det A=\sum_{\substack{r_{1},\ldots,r_{p} \\ 0\leq r_{a} \leq n_{a}}} \frac{1}{r_{1}!\ldots r_{p}!}\prod_{a=1}^{p} z_{a}^{n_{a}-r_{a}}   \sum_{\delta\in K(R_{\ul{r}})} \prod_{c\preccurlyeq \delta} -W(-A,c).
\end{equation}

\subsection{Fourth step : a combinatorial lemma} The last step of the proof consists in getting rid of the sets $R_{\ul{r}}$. This will be done by an application of Lemma \ref{lem:counting} below.

Let us choose a $p$-tuple $\ul{r}=(r_{1},\ldots,r_{p})$ with $0\leq r_a \leq n_a$ for all $a\in \set{p}$, and introduce the following subset of $\CMSf$:
\[\CMS_{\ul{r}}=\{\Cyc \text{ multiset of cyclic walks } : \forall a\in \set{p}, \,  v_{a}(\Cyc)= r_{a}\}.\]

Consider an element $\delta$ of $K(R_{\ul{r}})$. To each cycle $c=(i_{1}\cdots i_{r})$ of $\delta$, let us associate the cyclic walk $\bl(c)=\cy \bl(i_{1}),\ldots,\bl(i_{r}) \yc$. It is possible that two distinct cycles of $\delta$ yield the same cyclic walk, so that by doing this operation on each cycle of $\delta$ and collecting the resulting cyclic walks, we produce a multiset of cyclic walks, that is in fact an element of~$\CMS_{\ul{r}}$, that we call $f(\delta)$. In doing this, we just defined a map 
\[f:K(R_{\ul{r}})\to \CMS_{\ul{r}}.\]

\begin{lemma}\label{lem:counting} The map $f:K(R_{\ul{r}})\to \CMS_{\ul{r}}$ is onto and for every $\Cyc\in \CMS_{\ul{r}}$, we have 
\[|f^{-1}(\Cyc)|=\frac{r_{1}! \ldots r_{p}!}{\Cyc! \prod_{\cyc \in \Cyc} \val(\cyc)} .\] 
\end{lemma}

Before proving this lemma, let us explain how it allows us to conclude the proof of Theorem~\ref{thm:main}. In view of \eqref{eq:defhAcyc}, \eqref{eq:defhAcycperm}, and the definition of $f$, we have
\[\prod_{c\preccurlyeq \delta} -W(-A,c) =\prod_{c\preccurlyeq \delta} -W(-A,\bl(c)) = \prod_{\cyc\in f(\delta)} -W(-A,\cyc).\]
Accordingly, 
\[\sum_{\delta\in K(R_{\ul{r}})} \prod_{\cyc \in f(\delta)} -W(-A,\cyc)=\sum_{\Cyc\in \CMS_{\ul{r}}} |f^{-1}(\Cyc)| \prod_{\cyc \in \Cyc} -W(-A,\cyc).\]
Replacing in \eqref{eq:option2} and using Lemma \ref{lem:counting}, we find
\begin{equation}\label{eq:option3}
\det A=\sum_{\substack{r_{1},\ldots,r_{p} \\ 0\leq r_{a} \leq n_{a}}} \prod_{a=1}^{p} z_{a}^{n_{a}-r_{a}}   
\sum_{\Cyc\in \CMS_{\ul{r}}} \frac{1}{\Cyc!\prod_{\cyc\in \Cyc} \val(\cyc)} \prod_{\cyc \in \Cyc} -W(-A,\cyc),
\end{equation}
which is equivalent to Theorem \ref{thm:main}.

\subsection{Proof of the combinatorial lemma}\label{sec:proofcounting}

In this last step, we prove Lemma \ref{lem:counting}. We will call {\em list} of elements of a set $X$ an element of $\bigcup_{k\geq 1}X^{k}$.

Let us consider $\Cyc\in \CMS_{\ul{r}}$. Our first claim is that $\Cyc$ admits a preimage by~$f$. To prove this, let us order totally $\Cyc$, which is a multiset of cyclic walks, to produce a list of cyclic walks, in which repetitions may occur, according to multiplicity. Then, for each cyclic walk in this list, let us break the symmetry under cyclic permutation by choosing a base point, thus producing a list of elements of $\set{p}$. From $\Cyc$, we have now constructed a list of lists of elements of $\set{p}$, which is of the form
\[((a_{1,1},\ldots,a_{1,r_{1}}),(a_{2,1},\ldots,a_{2,r_{2}}),\ldots,(a_{p,1},\ldots,a_{p,r_{p}})).\]
In this list of lists, there are $r_{1}$ occurrences of $1$, $r_{2}$ of $2$, and so on, and $r_{p}$ occurrences of $p$. For each $a\in \set{p}$, let us replace in an arbitrary way the $r_{a}$ occurrences of $a$ by the $r_{a}$ elements of $I_{a}\cap R_{\ul{r}}$. The result is a list of lists of elements of $R_{\ul{r}}$ in which every element of $R_{\ul{r}}$ appears exactly once. Let us look at this list as the list of cycles of a permutation of $R_{\ul{r}}$, that we denote by $\delta$. Then, $\delta$ is an element of $K(R_{\ul{r}})$, and $f(\delta)=\Cyc$.

In fact, any preimage of $\Cyc$ can be obtained in this way, and we will now count in how many distinct ways this can be done.

Let us introduce the group
\[B_{\ul{r}}=\prod_{a=1}^{p} {\rm Perm}(I_{a}\cap R_{\ul{r}})\]
of block-preserving permutations of $R_{\ul{r}}$. Note that $|B_{\ul{r}}|=r_{1}!\ldots r_{p}!$. This group acts by conjugation on $K(R_{\ul{r}})$, and the orbits of this action are exactly the fibres of the map $f$. 

Thus, to compute the size of the preimage of $\Cyc$, we need to compute the size of the orbit of~$\delta$ under the action by conjugation of $B_{\ul{r}}$. This is equal to 
\begin{equation}\label{eq:preimage}
|f^{-1}(\Cyc)|=\frac{|B_{\ul{r}}|}{|Z(\delta)\cap B_{\ul{r}}|},
\end{equation}
 where $Z(\delta)$ is the centraliser of $\delta$, and $Z(\delta)\cap B_{\ul{r}}$ is the stabiliser of $\delta$ for the action of $B_{\ul{r}}$. 

To compute the cardinality of $Z(\delta)\cap B_{\ul{r}}$, let us write $\delta$ as a product of cycles:
\[\delta=c_{1,1}\ldots c_{1,m_{1}}\ldots c_{q,1}\ldots c_{q,m_{q}},\]
where two cycles $c_{i,j}$ and $c_{i',j'}$ induce the same cyclic walk on $\set{p}$ if and only if $i=i'$. Thus,~$q$ is the number of distinct elements of $\Cyc$ and the integers $m_{1},\ldots,m_{q}$ are their multiplicities. In particular, $m_{1}!\ldots m_{q}!$ is by definition the number $\Cyc!$. 

Let us consider an element $\sigma$ of $Z(\delta)\cap B_{\ul{r}}$. Since $\sigma$ commutes to $\delta$, it sends the support of any cycle of $\delta$ to the support of a cycle of $\delta$ of the same length. Moreover, since $\sigma$ belongs to~$B_{\ul{r}}$, any cycle of $\delta$ and its image by $\sigma$ induce the same cyclic walk. Thus, there is a group homomorphism 
\[Z(\delta)\cap B_{\ul{r}} \to \S_{m_{1}}\times \ldots \times \S_{m_{q}}.\]

Let $(i_{1}\cdots i_{s})$ and $(j_{1}\cdots j_{s})$ be two cycles of $\delta$ that induce the same cyclic walk. We may assume that they are written in such a way that $\bl(i_{k})=\bl(j_{k})$ for each $k\in \set{s}$. Then, the permutation $\sigma=(i_{1}\, j_{1})\ldots (i_{s}\, j_{s})$ belongs to $Z(\delta)\cap B_{\ul{r}}$ and is sent by the morphism above to the transposition that exchanges our two cycles. Thus, this morphism is onto. Let us describe and enumerate its kernel.

Let $\sigma$ be an element of this kernel. Let $c=(i_{1}\cdots i_{s})$ be a cycle of $\delta$. Then $\sigma$ leaves the support of $c$ globally stable. Within this support, it is a permutation that commutes with $c$, and therefore is a power of $c$. Thus, counting modulo $s$ for indices, we have $\sigma(i_{k})=i_{k+l}$ for some integer $l$. The condition that $\sigma$ belongs to $B_{\ul{r}}$ imposes that $\bl(i_{k})=\bl(i_{k+l})$ for each $k$, so that the cyclic shift of order $l$ leaves the cyclic walk $\cyc=\cy \bl(i_{1}),\ldots,\bl(i_{s}) \yc$ invariant. The set of~$l$'s for which this is true has, by definition, a cardinality equal to the valuation of $\cyc$.

To summarize, there is an exact sequence
\[1\to \prod_{c\preccurlyeq \delta} \big(\langle c \rangle \cap B_{\ul{r}}\big) \to Z(\delta)\cap B_{\ul{r}} \to \S_{m_{1}}\times \ldots \times \S_{m_{q}} \to 1\]
where for each cycle $c$, we have $|\langle c \rangle \cap B_{\ul{r}}|=\val(\cyc)$,  and we find
\[|Z(\delta)\cap B_{\ul{r}}|=\Cyc! \, \prod_{\cyc \in \Cyc} \val(\cyc).\]
Combined with \eqref{eq:preimage}, this completes the proof of Lemma \ref{lem:counting}, and the proof of Theorem \ref{thm:main}.

\subsection{An integral version}\label{sec:thmmainintegral}
With the case of matrices with complex entries in mind, we wrote the above proof of Theorem \ref{thm:main} without trying to avoid divisions by integers. However, such divisions occurred only at a few places, and could have been avoided entirely.\footnote{This is an asset of this combinatorial proof compared to the transcendental one given in Section \ref{sec:alternative-proof}.} The main point, which follows from Lemma \ref{lem:counting}, is that for any multiset of cyclic walks $\Cyc \in \CMSf$, the integer $\Cyc!\prod_{\cyc\in \Cyc} \val(\cyc)$ divides $\prod_{a\in \set{p}}v_{a}(\Cyc)!$, and therefore $n_{1}!\ldots n_{p}!$. Thus, the following variant of our main result:
\begin{equation}
n_{1}!\ldots n_{p}!\det A= \sum_{\Cyc \in \CMSf} \frac{n_{1}!\ldots n_{p}!}{\Cyc!\prod_{\cyc\in \Cyc} \val(\cyc)}\ul{z}^{\ul{n}-\ul{v}(\Cyc)}  \prod_{\cyc\in \Cyc} -W(-A,\cyc)
\label{eq:dettrblockdiagintegral}
\end{equation}
involves only integer coefficients, and holds true for a matrix $A$ with entries in an arbitrary commutative ring.


\section{A transcendental approach}\label{sec:alternative-proof}

In this section, we give a second, shorter proof of Theorem \ref{thm:main}, and a then variant of this result in the style of an infinite Euler product factorization. Some of the content of this section is reminiscent of Reutenauer--Sch\"{u}tzenberger's proof \cite{RS} of Amitsur's identity for the determinant of a sum of matrices \cite{Amitsur} using Lyndon words and infinite product factorizations. 

The proof given in Section \ref{sec:main} uses only finite computations, holds for matrices with entries in an arbitrary commutative ring and, we believe, gives important insights into the combinatorics of the problem. 

The proof presented below relies on the infinitely many polynomial identities bundled in the identity
\begin{equation}\label{eq:jacobi}
\det = \exp \Tr \log
\end{equation}
that we already mentioned in \eqref{eq:exptrlog1}.
It is much shorter than the previous one, but it is valid only for matrices with entries in a field of characteristic $0$.

\subsection{A second proof of Theorem \ref{thm:main}}

Let $A$ be a block matrix of the same shape as the one that we considered in Section \ref{sec:formblockmat}, now with coefficients in a field of characteristic $0$. Let us make the assumption that the matrix $A$ has scalar block diagonal coefficients $A_{[aa]}=z_{a}I_{n_{a}}$. 
 
Our goal is to establish a formula for $\det A$, which is a polynomial function of $z_{1},\ldots,z_{p}$. However, it turns out to be more convenient to compute it as a Laurent polynomial in $z_{1},\ldots,z_{p}$, that is, to allow ourselves to consider the multiplicative inverses $z_{1}^{-1},\ldots,z_{p}^{-1}$. The obvious drawback is that some steps of the computation become meaningless in the case where some of the scalars $z_{1},\ldots,z_{p}$ are zero. Fortunately, over a field of characteristic $0$, which is infinite, a polynomial in $z_{1},\ldots,z_{p}$ is uniquely determined by its values on the set $\{z_{1}\neq 0, \ldots, z_{n}\neq 0\}$, and the formula that we find for $\det A$ holds for all values of $z_{1},\ldots,z_{p}$.
  
Let us define the three block matrices
 \[Z=\diag(z_{a}I_{n_{a}} : a\in \set{p}) , \  \ B=A-Z , \ \text{ and } C=Z^{-1}B.\]
Then $Z(I_{n}+C)=Z+B=A$, so that the following equality of Laurent polynomials in $z_{1},\ldots,z_{p}$ holds:
\[\det A = \prod_{i=1}^{p} z_{a}^{n_{a}} \, \det (I_{n}+C)\,.\]

We will now compute $\det(I_{n}+C)$, and we will first do it as a formal series in $z_{1}^{-1},\ldots,z_{p}^{-1}$. The first step of the computation is
\begin{equation} \label{eq:detexptrlog}
\det (I_n+ C)  = \exp \Tr \log \left( I_n + C \right) = \exp \bigg(-\Tr \sum_{k=1}^{\infty} \frac{(-C)^k}{k}\bigg).
\end{equation}

Recall from Section \ref{sec:cyclicwalks} the definition of a cyclic walk, its length and its valuation. For any integer $k\geq 2$, we have
\[\Tr(C^{k})=\sum_{a_{1}\in \set{p}} \Tr\big((C^{k})_{[a_{1}a_{1}]}\big)=\sum_{a_{1},\ldots,a_{k}\in \set{p}} \Tr(C_{[a_{1}a_{2}]}\ldots C_{[a_{k}a_{1}]})\]
and since the diagonal blocks of $C$ are zero, the terms in which two consecutive indices $a_{i}$ are equal (including $a_{k}$ and $a_{1}$) do not contribute. Therefore,
\[\Tr(C^{k})=\sum_{\substack{a_{1},\ldots,a_{k}\in \set{p}\\ a_{1}\neq a_{2}\neq \cdots \neq a_{k}\neq a_{1}}} W\big(C,\cy a_{1},\ldots,a_{k}\yc\big).\]
In this sum, a cyclic walk $\cyc=\cy a_{1},\ldots,a_{k}\yc$ appears as many times as there are distinct cyclic permutations of $a_{1},\ldots,a_{k}$, that is, ${k}/{\val(\cyc)}$. Thus, for all $k\geq 2$,
\[\frac{\Tr(C^{k})}{k}=\sum_{\cyc \in \cW, |\cyc|=k} \frac{W(C, \cyc)}{\val(\cyc)}\,.\]
For $k=1$, the right-hand side is zero, because cyclic walks have length at least $2$, and the right-hand side is zero because the diagonal of $C$ vanishes. Therefore, the identity still holds in this case. Thanks to this observation, we have
\begin{equation}
\label{eq:expW}
\det (I_n+ C)  =\exp \ \sum_{\cyc \in \cW} \frac{-W(-C,\cyc)}{\val(\cyc)}\, .
\end{equation}
Let us emphasize that in the sum, the term indexed by a cyclic walk $\cyc$ is a homogeneous polynomial of degree $|\cyc|$, the length of $\cyc$, in $z_{1}^{-1},\ldots,z_{p}^{-1}$. The least degree of a term of the sum is $2$, and there are only finitely many terms of degree less than any given integer. 

We will now use the equality 
\[W(-C,\cyc)=\ul{z}^{-\ul{v}(\cyc)}W(-A,\cyc),\]
valid for any cyclic walk $\cyc$, and in which we stress that $W(-A,\cyc)$ has degree $0$ in $z_{1},\ldots,z_{p}$, because it is a sum of traces of products of off-diagonal blocks of $A$.

Expanding the exponential, we find
\begin{equation}\label{eq:series-pol}
\det (I_n+C)= 1+\sum_{\ell=1}^\infty \frac{1}{\ell!} \sum_{\cyc_{1},\ldots,\cyc_{\ell}\in \cW} \prod_{j=1}^{\ell} \bigg( \ul{z}^{-\ul{v}(\cyc_{j})} \frac{W(-A,\cyc_{j})}{\val(\cyc_{j})} \bigg)\,.
\end{equation}

The term of the sum indexed by the $\ell$-tuple $(\cyc_{1},\ldots,\cyc_{\ell})$ depends only on the number of occurrences of each cyclic walk, that is, on the underlying multiset of cyclic walks. The contribution of a multiset $\Cyc$ is equal to
\[ \ul{z}^{-\ul{v}(\Cyc)} \prod_{\cyc\in\Cyc} \frac{-W(-A,\cyc)}{\val(\cyc)}.\]
Besides, a multiset $\Cyc$ of cardinality $\ell$ can be obtained from ${\ell!}/{\Cyc!}$ distinct $\ell$-tuples of cyclic walks. Therefore, denoting by $\CMS$ the set of all multisets of cyclic walks, we have established the equality
\begin{equation}\label{eq:I+Cavant}
\det (I_n+C)= \sum_{\Cyc\in \CMS}  \frac{\ul{z}^{-\ul{v}(\Cyc)} }{\Cyc!}\prod_{\cyc\in\Cyc} \frac{-W(-A,\cyc)}{\val(\cyc)}\, ,
\end{equation}
with the natural convention that the contribution of the empty multiset is $1$.

On the right-hand side of this equality, we have a formal series in $z_{1}^{-1},\ldots,z_{p}^{-1}$.
On the left-hand side, we have a polynomial in the same variables with degree at most $n_a$ in the variable~$z_a^{-1}$, for each $a\in \set{p}$. We can therefore remove all the terms of higher order on the right-hand side, to find that 
\begin{equation}\label{eq:I+Capres}
\det (I_n+ C)  =  \sum_{\Cyc\in \CMSf} \frac{\ul{z}^{-\ul{v}(\Cyc)} }{\Cyc!} \prod_{\cyc\in\Cyc} \frac{-W(-A,\cyc)}{\val(\cyc)}\,.
\end{equation}
Multiplying by the determinant of $Z$, we find
\[\det A = \ul{z}^{\ul{n}} \det(I_n+ C) =  \sum_{\Cyc\in \CMSf} \frac{\ul{z}^{\ul{n}-\ul{v}(\Cyc)}}{\Cyc!} \prod_{\cyc\in\Cyc} \frac{-W(-A,\cyc)}{\val(\cyc)}\,,\]
which, according to a remark made at the beginning of this section, concludes this proof of Theorem \ref{thm:main}.

\subsection{Prime cyclic walks and an Euler product identity}\label{sec:primeeuler}

Let us call a cyclic walk of valuation $1$ a {\em prime} cyclic walk, and let $\cP$ be the set of prime cyclic walks. There is no canonical way of defining the concatenation of two cyclic walks, because cyclic walks do not have a base point, but it is possible to define the powers of a cyclic walk : if $\cyc=\cy a_{1},\ldots,a_{k}\yc$, then the cyclic walk 
\[\cy a_{1},\ldots,a_{k},a_{1},\ldots,a_{k}\yc=\cy a_{i},\ldots,a_{k},a_{1},\ldots,a_{i-1},a_{i},\ldots,a_{k},a_{1},\ldots,a_{i-1}\yc\]
does not depend on the choice of the representation of $\cyc$, and it is denoted $\cyc^{2}$. Of course, the $m$-th power $\cyc^{m}$ is equally well defined for any integer $m\geq 1$.

With this definition, any cyclic walk $\cyc$ can be written in a unique way as $\cyc=\pcyc^{m}$, with $\pcyc$ a prime cyclic walk and $m\geq 1$ an integer, which is also the valuation of $\cyc$. Therefore, starting again from \eqref{eq:expW}, we have

\begin{align*}
\det (I_n+ C)= \exp \bigg( -\sum_{\pcyc\in \cP} \sum_{m=1}^\infty \frac{W(-C,\pcyc^m)}{m} \bigg) 
=\prod_{\pcyc\in \cP}  \exp \bigg( - \sum_{m=1}^\infty \frac{W(-C,\pcyc^m)}{m} \bigg).
\end{align*}

For each prime cyclic walk $\pcyc=\cy a_{1},\ldots,a_{k}\yc$, reading \eqref{eq:detexptrlog} from right to left, we see that the term indexed by $\pcyc$ in the product above is equal to 
\begin{align*}
\exp \bigg( - \Tr \sum_{m=1}^\infty \frac{\big((-C)_{[a_{1}a_{2}]}\ldots (-C)_{[a_{k}a_{1}]}\big)^{m}}{m} \bigg)&=
\det\big(I_{n_{1}}-(-C)_{[a_{1}a_{2}]}\ldots (-C)_{[a_{k}a_{1}]}\big)\\
&=\det\big(I_{n_{1}}-\ul{z}^{-\ul{v}(\pcyc)}(-A)_{[a_{1}a_{2}]}\ldots (-A)_{[a_{k}a_{1}]}\big).
\end{align*}
We already observed at the occasion of the definition of $W(A,\cyc)$ (see \eqref{eq:defhAcyc}) that the matrix $A_{[a_{1}a_{2}]}\ldots A_{[a_{k}a_{1}]}$ is ill-defined, yet has a well-defined trace. In fact, its characteristic polynomial is well defined. Denoting this ill-defined matrix by $\hol(-A,\pcyc)$, the polynomial in $z_1^{-1},\ldots,z_p^{-1}$ that we are computing is equal to 
\[\det\big(I-\ul{z}^{-\ul{v}(\pcyc)}\, \hol(-A,\pcyc)\big).\]

With this notation, we find that the following result holds.

\begin{proposition} \label{prop:Euler}
With the notation of Theorem \ref{thm:main}, the following identity holds:
\[\det A= \ul{z}^{\ul{n}} \prod_{\pcyc\in \cP} \det\big(I-\ul{z}^{-\ul{v}(\pcyc)}\hol(-A,\pcyc)\big)\]
as an equality of Laurent series in $z_{1}^{-1},\ldots,z_{p}^{-1}$.
\end{proposition}

Let us make two further comments about this result. The first is about the definition of the product. 

Consider a prime cyclic walk $\pcyc$. The determinant $\det (I-t \hol(-A,\pcyc))$ is a polynomial in~$t$ which does not depend on the choice of a base point that is necessary to specify the matrix $\hol(-A,\pcyc)$. It has degree at most $n$, constant coefficient $1$, and we can name its other coefficients, by setting
\[\det \big(I-t \hol(-A,\pcyc)\big) = 1+ \sum_{j=1}^{n} \alpha_{j}(\pcyc) t^{j}.\]
With this notation, the term of the product indexed by $\pcyc$ is
\[1+\sum_{j=1}^{n} \alpha_{j}(\pcyc) z_{1}^{-jv_{1}(\pcyc)} \ldots z_{p}^{-jv_{p}(\pcyc)}.\]
There are only finitely many cyclic walks, a fortiori prime cyclic walks, with a given number of visits at each element of $\set{p}$. Thus, the product is well defined as a formal series in $z_{1}^{-1},\ldots,z_{p}^{-1}$. 

The second comment is about the possibility of removing some terms of the product, as we did when we went from \eqref{eq:I+Cavant} to \eqref{eq:I+Capres}. Indeed, $\det A$ is a polynomial in $z_{1},\ldots,z_{p}$, of degree $\ul{n}$. A cyclic walk $\pcyc$ such that $v_{a}(\pcyc)>n_{a}$ for some $a\in \set{p}$ will only contribute terms which have negative degree in at least one of the variables $z_{1},\ldots,z_{p}$. Nevertheless, one cannot remove the corresponding terms from the product, as they collectively contribute to cancelling higher order monomials which appear from determinants which also produce admissible terms.

Let us illustrate this in the case where $p=n$ and all blocks have size $1$. In this case, Proposition \ref{prop:Euler} reads
\[\det A=A_{11}\ldots A_{nn} \prod_{\pcyc=\cy i_{1}\ldots i_{r}\yc \in \cP} \bigg(1+(-1)^{r-1}\frac{A_{i_{1}i_{2}}\ldots A_{i_{r}i_{1}}}{A_{i_{1}i_{1}}\ldots A_{i_{r}i_{r}}}\bigg).\]
In the product, consider the terms that correspond to cycle of permutations, that is, to cyclic walks that do not visit twice any vertex. There is a finite number of such terms, and expanding their product produces all the terms of the classical definition of the determinant. However, it also produces, when one multiplies \emph{overlapping} cycles, terms of higher negative degree. These terms are cancelled by terms coming from cyclic walks that are not cycles of permutations. Similar cancellations occur until the whole product is taken into account. 

We will see below however, that when the underlying graph structure of the blocks is a quiver with only a finite number of prime cycles, then the only terms in the product not equal to $1$ are finite, and thus we can indeed truncate the product in that case, see Theorem  \ref{thm:finite-euler}.

\section{Wilson loop expansion of the twisted Laplacian determinant}\label{sec:wilson-loops}

Let $\rG=(\rV,\rE)$ be a finite quiver, as described in Section \ref{sec:graphsandlaplacian}. Let us make the assumption that the graph has no self-loops, that is, that for each edge $e\in \rE$, we have $s(e)\neq t(e)$.\footnote{It is possible that this assumption may be dropped by modifying the results appropriately, but for simplicity, we work with this assumption.} For simplicity, we assume that $\rV=\set{p}$.

Let $\ul{x}=(x_e)_{e\in \rE}$ be complex weights on those oriented edges. For each vertex~$v\in \rV$, let $n_v\ge 1$ be a positive integer, and let $\ul{U}=(U_e)_{e\in \rE}$ be complex matrices indexed by edges, where for each edge $e$ joining vertex $u$ to vertex $v$, the matrix $U_e$ is of size $n_u\times n_v$. It is the matrix of an endomorphism from $\C^{n_{v}}$ to $\C^{n_{u}}$.

Recall the definition of the covariant Laplacian $\Delta$, see \eqref{eq:covLaplacian}. Let us define $n=\sum_{v\in \rV} n_{v}$, so that $\Delta$ is an endomorphism of $\C^{n}=\bigoplus_{v\in \rV} \C^{n_{v}}$. Its matrix in the canonical basis of $\C^{n}$ has a block structure with $p=|\rV|$ vertices, and block sizes $(n_{v})_{v\in \rV}$. For each $u\in \rV$, let us define
\begin{equation}\label{eq:defz}
z_{u}=\sum_{\substack{e\in \rE \\ s(e)=u}} x_{e}.
\end{equation}

The $(u,v)$ block of the matrix of $\Delta$ is equal to
\begin{equation}\label{eq:coefDelta}
\Delta_{[uv]}=\left\{\begin{array}{ll}\displaystyle \phantom{-} z_{u} I_{n_{u}} & \text{if } u=v,\\[2mm]
\displaystyle -\sum_{e:u\to v} x_{e} U_{e}& \text{if } u\neq v.
\end{array}
\right.
\end{equation}

Let us say that a cyclic walk $\cy v_{1},\ldots,v_{k} \yc$ is {\em compatible with $\rG$} if for all $i=\{1,\ldots,m\}$, the graph $\rG$ has at least one edge from $v_{i}$ to $v_{i+1}$ (with the convention $v_{m+1}=v_{1}$). Let us denote by $\CMS^{\rG}$ the subset of $\CMS$ formed by multisets of cyclic walks compatible with $\rG$. 

Since there may be more than one edge joining two vertices of $\rG$, a cyclic walk compatible with $\rG$ can correspond to several distinct cycles on $\rG$, in the sense of well-chained sequences of edges up to cyclic permutation. We define $\CMS(\rG)$ as the set of multisets of cycles on $\rG$. There is a map $f$ from the set of cycles on $\rG$ to the set of cyclic walks on $\rV$ compatible with $\rG$ which forgets the edges traversed, and only remembers the vertices visited. This map is onto, and in general, many-to-one. We use the same notation for the induced map $f:\CMS(\rG)\to \CMS^{\rG}$.

For a cycle $\cyc$ in $\rG$, which traverses the edges $e_{1},\ldots,e_{k}$ in this cyclic order, we set 
\begin{equation}\label{eq:defhol}
\hol(\cyc)=U_{e_{1}}\ldots U_{e_{k}}.
\end{equation}
The same ambiguity holds here as in the definition \eqref{eq:defhAcyc}: the matrix $\hol(\cyc)$ depends on the choice of a basepoint for the cycle $\cyc$, but its trace is well defined.

For the same cycle $\cyc$, we define $\ul{x}^{\ul{e}(\cyc)}=x_{e_{1}}\ldots x_{e_{k}}$.

\begin{lemma} \label{lem:forget} For every cyclic walk $\check\cyc$ on $\rV$, we have  
\[W(-\Delta,\check\cyc)= \sum_{\substack{\cyc \text{ cycle on } \rG \\ f(\cyc)= \check\cyc}}  \ul{x}^{\ul{e}(\cyc)} \Tr\hol(\cyc).\]
\end{lemma}

\begin{proof} The left-hand side is defined by \eqref{eq:defhAcyc}. Replacing the blocks of $\Delta$ by their definition \eqref{eq:coefDelta} and expanding the product of sums, we arrive at a sum of traces that we recognize thanks to~\eqref{eq:defhol}.
\end{proof}

Given a multiset $\Cyc\in \CMS(\rG)$,  we keep the notation $\Cyc!$ for the product of the factorials of the multiplicities of the elements of $\Cyc$, and $\ul{v}(\Cyc)$ for the vector of number of visits of the elements of $\Cyc$ at each vertex of $\rG$. Similarly, for an element $\cyc\in \Cyc$, we keep the notation $\val(\cyc)$ for its valuation. We introduce the notation $\ul{e}(\Cyc)$ \label{def:e(C)}as the vector indexed by~$\rE$, the component corresponding to an edge being the number of times where this edge is traversed by an element of $\Cyc$, taking multiplicity into account.

\begin{theorem} \label{thm:polcarDelta} Let $T=\diag(t_{u}I_{n_{u}}:u\in \rV)$. The following equality holds:
\begin{equation*}
\det(T+\Delta)=\sum_{\ul{0}\leq \ul{k}\leq \ul{n}} \ul{t}^{\ul{k}}\,  \sum_{\substack{\Cyc \in \CMS(\rG) \\ \ul{v}(\Cyc)+\ul{k}\leq \ul{n}}}
\binom{n_{1}-v_{1}(\Cyc)}{k_{1}}\ldots \binom{n_{p}-v_{p}(\Cyc)}{k_{p}}
 \frac{\ul{z}^{\ul{n}-\ul{k}-\ul{v}(\Cyc)} }{\Cyc!} \, \ul{x}^{\ul{e}(\Cyc)}
 \prod_{\cyc\in \Cyc} \frac{-\Tr \hol(\cyc)}{\val(\cyc)}.
 \end{equation*}
 \end{theorem}

\begin{proof} We apply Corollary \ref{coro:main} and use Lemma \ref{lem:forget} for each term of the product.
\end{proof}

Setting $T=0$ yields the determinant of $\Delta$.

\begin{corollary} \label{coro:detDelta}
We have
\begin{equation}
\det \Delta= \sum_{\Cyc \in\CMS(\rG)_{\leq \ul{n}}} \frac{\ul{z}^{\ul{n}-\ul{v}(\Cyc)} }{\Cyc!} 
 \prod_{\cyc\in \Cyc} \frac{-\ul{x}^{\ul{e}(\cyc)}\Tr \hol(\cyc)}{\val(\cyc)}\, .
\end{equation}
\end{corollary}

The Laplacian determinant can be seen as a partition function for matter in a gauge field in the context of lattice gauge theory \cite{KL1,KL7}. With Corollary \ref{coro:detDelta} above, we can express moments of this partition function as linear combinations of certain Wilson loop, that is, expectations of products of traces of holonomies along cycles. 

\begin{theorem}\label{thm:gauge-theory}
Consider any probability distribution on the quiver representation for fixed $\ul{n}$. Then, for any integer $k\ge 1$, we have
\begin{equation*}
    \bE \left[\left(\det\Delta\right)^k \right] = \sum_{\Cyc_1, \ldots, \Cyc_k \in \CMS(\rG)_{\le \ul{n}}}
    \underbrace{\left(\prod_{i=1}^k \bigg[\frac{\ul{z}^{\ul{n}-\ul{v}(\Cyc_i)} }{\Cyc_i!} 
 \prod_{\cyc_i\in \Cyc_i} \frac{-\ul{x}^{\ul{e}(\cyc_i)}}{\val(\cyc_i)}\bigg] \right)}_{\text{Combinatorial weight}} \underbrace{\bE \left[\prod_{i=1}^k \prod_{\cyc_i\in \Cyc_i}\Tr \hol(\cyc_i) \right]}_{\text{Wilson loop}}\,.
\end{equation*}
\end{theorem}


\section{Discrete vector fields interpretation}\label{sec:vector-fields}

In this section, we explain ways in which one can interpret the combinatorial weights appearing in the main results of Section \ref{sec:wilson-loops}. Our goal is to replace the monomials in $z_a$ by linear combinations of monomials in the weights $x_e$ and describe the new coefficients combinatorially.

To start off, let us warm-up by explaining the well-understood $\ul{n}=\ul{1}$ case. 

\subsection{A classical theorem of Forman}\label{sec:warmup}

Let us show how Corollary \ref{coro:detDelta} implies \cite[Eq. (1)]{Forman}. In the case where $p=n$, so that $n_1=\ldots=n_p=1$, Corollary \ref{coro:detDelta}  reads
\[\det \Delta = \sum_{\Cyc \in \CMS(\rG)_{\leq \ul{1}}} (-1)^{\vert \Cyc\vert} \hspace{-2mm} \prod_{\substack{v\in \rV \text{ not} \\ \text{covered by } \Cyc}}\hspace{-4mm} z_{v} \hspace{4mm}
 \prod_{\cyc\in\Cyc} \ul{x}^{\ul{e}(\cyc)} \hol(\cyc)\,.\]
Note that in that case, $ \CMS(\rG)_{\leq \ul{1}}$ consists in sets of disjoint and non-self-crossing cyclic walks.

Then, by a resummation argument ``\emph{à la Zeilberger}'' \cite{Zeilberger} that is explained in \cite{KL2}, we may rewrite this sum as
\begin{align}\label{eq:forman-zeilberger}
\det \Delta & = \sum_{\Cyc \in \CMS(\rG)_{\leq\ul{1}}}  
\bigg( \prod_{\substack{v\in \rV \text{ not} \\ \text{covered by } \Cyc}} \hspace{2mm} 
\sum_{\substack{e\in \rE\\s(e)=v}} x_{e} \bigg) \prod_{\cyc\in \Cyc} -  \ul{x}^{\ul{e}(\cyc)} \hol(\cyc)\nonumber \\
& = \sum_{F \in \cU} \ul{x}^{F} \prod_{\cyc\in \mathcal C (F)} \big(1-\hol(\cyc)\big)\,,
\end{align}
where $\cU$ is the set of subsets of $\rE$ containing exactly one edge starting from each vertex\footnote{If $\cU=\varnothing$, it means that $\rG$ has an isolated vertex, and thus $\det \Delta=0$, which is consistent with the sum being empty.} and for each~$F\in \cU$, we let $\mathcal C(F)$ is the set of simple cycles on $\rG$ that can be obtained using only the edges of~$F$.

This last equality is the matrix-tree theorem of Zaslavsky--Chaiken--Forman--Kenyon \cite{Zaslavsky, Chaiken, Forman, Kenyon} expressing the determinant of the covariant Laplacian as a sum over discrete vector fields of a weight proportional to the $1$ minus the trace of the holonomy of its limit-cycles.

Let us now consider the general $\ul{n}$ case.

\subsection{Discrete vector fields}

Let us introduce some notation. We define, for each $a\in \set{p}$, the set $\rE_{a}=\{e\in \rE : s(e)=a\}$ of edges starting from $a$, and, for each $\ul{m}\leq \ul{n}$, the set
\begin{equation}\mathscr E_{\ul{m}}=\{X \text{ multiset of elements of } \rE : \forall a\in \rV, \|X\cap \rE_{a}\|=m_a\}
\end{equation}
of $\ul{m}$-vector fields on $\rG$. Here and thereafter, we use the notation $\|\cdot\|$ for the cardinality with multiplicity of a multiset.

Let us define the set of ordered stacks of edges
\[\Xi=\big\{\xi:\set{n} \to \rE \ \text{ such that } \ \forall i\in \set{n}, \ \xi(i)\in \rE_{\bl(i)}\big\}\,.\]

For each $\xi\in \Xi$, we define the set of well-chained permutations of $\xi$:
\[\Sigma(\xi)=\big\{\sigma\in \S_{n} : \forall i\in\set{n}, \, \sigma(i)=i \text{ or }t(\xi(i))=\bl(\sigma(i))\big\}.\]
Given $\sigma\in \Sigma(\xi)$, we let $C(\sigma)$ be the set of its cycles of length at least $2$.

\begin{theorem}\label{thm:formanN}\label{thm:forman-zeilberger-vector-fields}
We have
\begin{equation}\label{eq:vector-fields}
\det \Delta= \frac{1}{n_{1}!\ldots n_{p}!}\sum_{\xi\in \Xi} \ul{x}^{\xi} \bigg[\sum_{\sigma \in \Sigma(\xi)} \prod_{a=1}^{p}  (n_{a}-v_{a}(\sigma))!
\prod_{\cyc \in C(\sigma)} -\Tr \hol(\cyc)\bigg]\,.
\end{equation}
\end{theorem}

As already explained above in Subsection \ref{sec:warmup}, \eqref{eq:vector-fields} specializes to a theorem of Forman \cite[Eq. (1)]{Forman} when all $n_a$ are equal to $1$.

\begin{proof} 
We start with Corollary \ref{coro:detDelta}, that is
\[\det \Delta= \sum_{\Cyc \in\CMS(\rG)_{\leq \ul{n}}} \frac{\ul{z}^{\ul{n}-\ul{v}(\Cyc)} }{\Cyc!} 
 \prod_{\cyc\in \Cyc} \frac{-\ul{x}^{\ul{e}(\cyc)}\Tr \hol(\cyc)}{\val(\cyc)}\,,\]
and our goal is to replace the monomials in $z_a$ by linear combinations of monomials in the weights $x_e$ and describe the new coefficients.

In order to do so, recall from \eqref{eq:defz} the definition of $z_{a}$ for $a\in \set{p}$. For each $\ul{m}$,
\[\ul{z}^{\ul{m}}=\prod_{a=1}^{p} z_{a}^{m_a}=\sum_{Y\in \mathscr E_{\ul{m}}} \frac{m_{1}!\ldots m_{p}!}{Y!}\, \ul{x}^{Y}.\]
Thus, we find
\begin{align*}
\det \Delta &
= \sum_{\Cyc \in\CMS(\rG)_{\leq \ul{n}}} 
\sum_{Y\in \mathscr E_{\ul{n}-\ul{v}(\Cyc)} }
\frac{1}{\Cyc!}  \frac{\prod_{a=1}^{p} (n_a-v_{a}(\Cyc))!}{Y!} \, \ul{x}^{Y\cup \Cyc}
\ \prod_{\cyc\in \Cyc} \frac{-\Tr \hol(\cyc)}{\val(\cyc)}\, .\\
\end{align*}

At this point, our goal is to invert the order of summation. 

Given $\Cyc\in \CMSf$, let $\rE(\Cyc)$ denote the multiset of elements of $\rE$ used in $\Cyc$.\footnote{The multiplicity of all edges $e\in \rE$ in $\rE(\Cyc)$ is the integer vector introduced above on p.\pageref{def:e(C)} and denoted $\ul{e}(\Cyc)$.} The condition $Y\in \mathscr E_{\ul{n}-\ul{v}(\Cyc)}$ is equivalent to $Y=X\setminus \rE(\Cyc)$ for some $X\in \mathscr E_{\ul{n}}$ such that $\rE(\Cyc)\subseteq X$. In the following, we let $\mathscr C(X)$ be the multisets of cycles of $\rG$ such that $\rE(\Cyc)\subseteq X$. Thus,
\begin{equation}\label{eq:multinomial}
\prod_{a=1}^{p} z_{a}^{n_a-v_{a}(\Cyc)}=\sum_{\substack{X\in \mathscr E_{\ul{n}} \\ \Cyc\in \mathscr C(X)}} \frac{\prod_{a=1}^{p} (n_a-v_{a}(\Cyc))!}{(X\setminus \rE(\Cyc))!}
\ul{x}^{X\setminus \rE(\Cyc)}\,.
\end{equation}

Thus, inserting \eqref{eq:multinomial} in the above computation and exchanging the sums, we find
\begin{equation}\label{eq:version1}
\det \Delta=\sum_{\substack{X\in \mathscr E_{\ul{n}}\\ \Cyc\in \mathscr C(X)}} \ul{x}^{X} \ \frac{1}{\Cyc! \prod_{\cyc\in\Cyc}\val(\cyc)}\ \frac{\prod_{a=1}^{p} (n_a-v_{a}(\Cyc))!}{(X\setminus \rE(\Cyc))!} \ \prod_{\cyc\in \Cyc} -\Tr\hol(\cyc)\, .
\end{equation}

We will now lift this sum to a sum over stacks of edges and well-chained permutations thereof. To $\xi\in \Xi$ and $\sigma\in \Sigma(\xi)$, we associate the multiset $X$ given by the range of $\xi$ counting multiplicities, and the multiset of cyclic walks $\Cyc$ given by projecting the elements of $C(\sigma)$ down to $\set{p}$. We write $(X,\Cyc)=F(\xi,\sigma)$ and we need to determine how many preimages by $F$ there exist for a given pair $(X,\Cyc)$.

Along the lines of the arguments leading to Lemma \ref{lem:counting}, we claim that the set of preimages of $(X,\Cyc)$ by $F$ is an orbit of the action of $\S_{n_{1}}\times \ldots \times \S_{n_{p}}$, and the stabilizer of a pair $(\xi,\sigma)$ has $(X\setminus \rE(\Cyc))! \Cyc! \prod_{\cyc\in\Cyc}\val(\cyc)$ elements. Therefore, 
\[\big|F^{-1}\big((X,\Cyc)\big)\big|=\frac{\prod_{a=1}^{p} n_{a}!}{(X\setminus \rE(\Cyc))!\Cyc! \prod_{\cyc\in\Cyc}\val(\cyc)}.\]
This factor is already present in \eqref{eq:version1}, so that, by a summation formula,
\begin{align*}
\det \Delta&=\sum_{\substack{X\in \mathscr E_{\ul{n}} \\ \Cyc\in \mathscr C(X)}} \big|F^{-1}\big((X,\Cyc)\big)\big|^{-1} \  \ul{x}^{X}\ \prod_{a=1}^{p} \frac{(n_a-v_{a}(\Cyc))!}{n_{a}!} \ \prod_{\cyc\in \Cyc} -\Tr\hol(\cyc)\\
&=\sum_{\substack{\xi \in \Xi \\ \sigma \in \Sigma(\xi)}}\ul{x}^{\xi}\ \prod_{a=1}^{p} \frac{(n_a-v_{a}(C(\sigma)))!}{n_{a}!} \ \prod_{c\in C(\sigma)} -\Tr\hol(c),\\
\end{align*}
which concludes the proof.
\end{proof}


\subsection{Reinterpretations of the combinatorial factor in Equation \eqref{eq:vector-fields}}\label{sec:reinterpretations}

We now provide two reinterpretations of the combinatorial factor $\prod_{a=1}^p (n_a-\vert \sigma\vert_a)!$ in \eqref{eq:vector-fields}.

\subsubsection{A first point of view}

Set
\[\Sigma'(\xi)=\big\{\sigma\in \S_{n} : \forall c=(i_{1}\cdots i_{r})\preccurlyeq \sigma, \bl(i_{1})=\ldots =\bl(i_{r}) \text{ or } \forall s\in\set{r}, \, \xi(i_{s})=\big(\bl(i_{s}),\bl(\sigma(i_{s}))\big)\big\}.\]
This is the set of permutations for which each cycle is either purely stationary (staying inside~$I_a$, for some $a$), or purely cinematic (never stepping twice consecutively in the same~$I_a$).
The previously defined set $\Sigma(\xi)$ is a subset of $\Sigma'(\xi)$ where the stationary cycles are all of length $1$ (that is, they are fixed points).

For $\sigma\in \Sigma'(\xi)$, we let $C'(\sigma)$ be the set of cycles of $\sigma$ that visit at least two distinct blocks. When $\sigma\in \Sigma(\xi)$, we have $C'(\sigma)=C(\sigma)$.

The number of ways to build an element of $\Sigma'(\xi)$ from an element $\sigma\in \Sigma(\xi)$ is the product of factorials $\prod_{a=1}^p (n_{a}-|\sigma|_{a})!$. Thus, Theorem \ref{thm:formanN} implies the following corollary.

\begin{corollary}We have
\begin{equation}\label{eq:sigmaprime}
\det \Delta=\frac{1}{n_{1}!\ldots n_{p}! } \, \sum_{\xi\in \Xi} \, \ul{x}^{\xi}  \bigg[\sum_{\sigma \in \Sigma'(\xi)} \, \prod_{\cyc\in C'(\sigma)} -\Tr\hol(\cyc)\bigg]\,.
\end{equation}
\end{corollary}

\subsubsection{A second point of view}
 
Set $B=\S_{n_{1}}\times \ldots \times \S_{n_{p}}$. For two elements $\alpha, \beta\in B$, we write 
\[\alpha \leq \beta\]
when all cycles of $\alpha$ of length at least $2$ are cycles of $\beta$. Note that with this definition, the identity permutation $1$ satisfies $1\leq \beta$ for all $\beta$.

Every $\sigma\in \Sigma(\xi)$ induces an element of $B$ and we let 
\[P:\Sigma(\xi)\to B\] 
denote the corresponding map. Adding a permutation of the fixed points of~$\sigma\in \Sigma(\xi)$ (the number of ways of doing that is the product of factorials $\prod_{a=1}^p (n_{a}-|\sigma|_{a})!$), we get an element of $B$.

Thus, Theorem \ref{thm:formanN} implies the following corollary.

\begin{corollary}We have
\begin{equation}\label{eq:eq29}
\det \Delta=\frac{1}{n_{1}!\ldots n_{p}! } \, \sum_{\substack{\xi\in \Xi\\ \beta \in B}} \, \ul{x}^{\xi} \bigg[\sum_{\substack{\sigma \in \Sigma(\xi) \\ P(\sigma)\leq \beta}} \,  \prod_{\cyc\in C(\sigma)} -\Tr\hol(\cyc) \bigg]\,.
\end{equation}
\end{corollary}


\section{Euler factorization of the twisted Laplacian determinant}\label{sec:euler-product}

As a byproduct of the second proof of Theorem~\ref{thm:main}, we obtained Proposition \ref{prop:Euler}. 
Let us reformulate this proposition in the case of a quiver representation, and also show that an identity holds numerically, and not just in the sense of formal Laurent series. As we noted in the comments following Proposition \ref{prop:Euler}, one cannot truncate the product, so in order to have a numerical identity, we need to show that the product is absolutely convergent.

In the case where the quiver has only a finite number of prime cycles, Proposition \ref{prop:Euler} yields an equality. 

\begin{theorem}\label{thm:finite-euler}
In the case where $\rG$ only has a finite number of prime cycles, we have 
\begin{equation}\label{eq:coro-finite}
    \det \Delta = \ul{z}^{\ul{n}} \prod_{\cyc\in \cP(\rG)} \det(I-p(\cyc)\hol(\cyc))\,,
\end{equation}
where $p_{e}={x_{e}}/{z_{s(e)}}$.
\end{theorem}
\begin{proof}
This theorem will appear as a corollary of Theorem \ref{thm:euler-product} below, since when the product is finite in the right-hand side of \eqref{eq:twisted-zeta}, we can set $\kappa_a=0$ for all $a$. 
\end{proof}

Note that since the prime cycles are necessarily disjoint when there are only a finite number, the right-hand side of \eqref{eq:coro-finite} is indeed a polynomial in $\ul{x}$, as it should be from its equality with the left-hand side.

In general, when $\cP(\rG)$ is infinite, we need to be careful about the convergence of the infinite product. We thus make some assumptions, as follows. 

Let us assume that for each edge $e$, we have $x_{e}\geq 0$. Also, for each $a\in \set{p}$, let $\kappa_{a}\geq 0$ be a real number. 
\begin{assumption}\label{assumption}
Let us assume that for each $a\in \set{p}$, there exists $b\in \set{p}$ with $\kappa_{b}>0$ such that there exists a path from $a$ to $b$ in $\rG$.
\end{assumption}

In probabilistic terms, this assumption means that the Markov process with transition $p$ is {\em strictly sub-Markovian}, in the sense that its trajectories almost surely eventually end.

For each edge $e\in \rE$, set
\begin{equation}
p_{e}=\frac{x_{e}}{z_{s(e)}+\kappa_{s(e)}}\,.
\end{equation}
Let us define blockwise a matrix $P$ by setting, for all $a,b\in \set{p}$,
\begin{equation}
P_{[ab]}=\sum_{s(e)=a, t(e)=b}p_{e} U_{e}\,.    
\end{equation}

\begin{assumption}\label{assumption2}
Let us assume that $\|U_{e}\|\leq 1$ for all $e$, for some associated norm. 
\end{assumption}

Together, Assumption \ref{assumption} and Assumption \ref{assumption2} imply that the spectral radius of $P$ is strictly less than~$1$. 

For any cycle $c=(a_1,\ldots, a_m)$, we recall that $\hol(\cyc)=U_{a_{1}a_{2}}\ldots U_{a_{m}a_{1}} \in M_{n_{a_1}}(\C)$.
Note that under Assumption~\ref{assumption}, given a cyclic walk $\cyc=\cy a_1,\ldots, a_m\yc$, the determinant 
\begin{equation}\label{eq:det-based}
\det\left(I_{n_{a_i}}-p(\cyc)\hol(c)\right)
\end{equation}
is independent of the choice of representative $c=(a_i,\ldots, a_m, a_1,\ldots, a_{i-1})$ of $\cyc$. 
This follows from the fact that $p(c)$ is independent of the representative $c$, and that $\hol(c)$ are all conjugate of one another, so that the result follows from the identity, valid for any rectangular matrices~$A$ and~$B$ of compatible dimensions:
\[\det (I+ AB)= \det(I+BA)\,.\]
With a slight abuse of notation, we thus denote $\det\left(I-p(\cyc)\hol_U(\cyc)\right)$ the common value of the determinants \eqref{eq:det-based}.

Let $\cP(\rG)$ denote the set of prime cycles on $\rG$. Note that these cycles are not necessarily simple and in general may have self-intersection.

\begin{theorem}\label{thm:euler-product}
Under Assumptions \ref{assumption} and \ref{assumption2}, we have 
\begin{equation}\label{eq:twisted-zeta}
    \det \big(\diag(\kappa_a:a\in\rV)+\Delta\big)= (\ul{z}+\ul{\kappa})^{\ul{n}} \prod_{\cyc\in \cP(\rG)} \det\big(I-\ul{p}^{-\ul{e}(\cyc)}\hol(\cyc))\big)\,.
\end{equation}
\end{theorem}

\begin{remark}
We noted that Equation \eqref{eq:vector-fields} generalizes Forman's \cite[Eq. (1)]{Forman}. Equation~\eqref{eq:twisted-zeta} is reminiscent of another formula of Forman, for zeta functions of discrete dynamical systems on graphs, see \cite[Theorem 7.5]{Forman-zeta}. This formula is analogous to the expression $\sum_{n\ge 1} n^{-s} = \prod_{p \,\text{prime}} (1-p^{-s})^{-1}$, which is the Euler expansion for the classical zeta function of number theory.

It is moreover reminiscent of the twisted zeta function introduced already by Selberg, and which appears for instance in \cite[Eq. (4.3) and (4.4)]{Ray-Singer} and \cite[Theorem 1]{Fried} and \cite[Eq. (1.1)]{Jorgenson-Smajlovic-Spilioti}. This zeta function is of the form 
\begin{equation}\label{eq:ray-singer}
    \prod_{\substack{p \, \text{prime}\\ \text{closed geodesic}}} \prod_{k=1}^\infty \det(I-e^{-k\ell(p)}\hol(p))\,,
\end{equation}
where $\ell(p)$ is the length of $p$. The fact that we do not have the second infinite product in \eqref{eq:twisted-zeta} probably accounts for the fact that the prime closed curves we consider are not geodesic (in the sense that they have backtracking steps) and that this zeta function is the one of Ihara which we do not consider here.
\end{remark}

\begin{proof}
Under Assumptions~\ref{assumption} and ~\ref{assumption2}, $P$ has spectral radius strictly less than $1$, and we can perform analytic calculus using~$P$ as a variable in series with convergence radius greater or equal to $1$. 
We find
\begin{align*}
(\ul{z}+\ul{\kappa})^{-\ul{n}}\det(\diag(\kappa_a:a\in\rV)+\Delta) & = \exp \Tr \log (I_n-P) = \exp \left( - \sum_{\cyc\in \cP(\rG)} \sum_{m=1}^\infty \frac{\Tr\hol(\cyc^m)}{m}\right) \\
&\hspace{-1cm} = \prod_{\cyc\in \cP(\rG)} \exp \left(-  \sum_{m=1}^\infty \frac{\Tr\hol(\cyc^m)}{m}\right) = \prod_{\cyc\in \cP(\rG)} \det \left(I - p(\cyc) \hol(\cyc)\right)\,.
\end{align*}
This concludes the proof.
\end{proof}

The advantage of Theorem \ref{thm:euler-product} over Theorem \ref{thm:formanN} is that it allows to show comparison inequalities for certain families of representations, namely unitary representations of constant rank.

For the rest of this subsection, assume that $\rG$ is a bidirected graph (in the sense of Subsection~\ref{sec:graphvsquivers}: that is, each edge $e$ has an inverse $e^{-1}$ also in $\rE$ and the weights $x_e$ are positive real numbers which are symmetric under inversion). Let us assume that $\rG$ is connected. Let $\Delta_0$ be the classical Laplacian of $\rG$ defined as the operator acting on functions by 
\begin{equation}
    \Delta_0 f(v)=\sum_{\substack{e\in \rE\\ s(e)=v}} x_e \left[f(s(e))-f(t(e))\right]\,.
\end{equation}
This is a self-adjoint nonnegative operator, with kernel of dimension $1$ (because $\rG$ is assumed to be connected), and its eigenvalues are 
\begin{equation}
\mu_1 \ge \mu_2 \ge \ldots \ge \mu_{p-1} >\mu_p=0\,.    
\end{equation}

\begin{corollary}Let $\rG$ be a bidirected finite connected weighted graph as above.
Let $N$ be an integer and $\ul{U}$ be a unitary representation of the graph $\rG$ endowed with a vector bundle of rank~$N$ (with $U_{e^{-1}}=U_e^{-1}$ for all $e\in \rE$).  
Let $\lambda_1\ge \ldots \ge \lambda_{Np} \ge 0$ be the spectrum of $\Delta$ associated with~$\ul{U}$. Then for all $t\geq 0$, we have
\begin{equation}
   \det(t+\Delta)= \prod_{i=1}^{N p} (t+\lambda_i) \ge  \prod_{i=1}^{p}( t+ \mu_i)^N=\det(t+\Delta_{0})^{N}\,.
\end{equation}
\end{corollary}
\begin{proof}
We apply Theorem \ref{thm:euler-product} to $t+\Delta$ whose eigenvalues are those of $\Delta$ shifted by $t$.
For each $\cyc \in \cP(\rG)$, let $u_i$ be the eigenvalues of $\hol(\cyc)$, which are unit complex numbers by unitary of the holonomy. 
Then, by the triangular inequality, we have 
\begin{equation*}
   \big| \det\big(I-\ul{p}^{\ul{e}(\cyc)} \hol(\cyc)\big) \big|= \prod_{i=1}^N \big| 1-\ul{p}^{\ul{e}(\cyc)} u_i\big| \ge \big(1-\ul{p}^{\ul{e}(\cyc)} \big)^{N}\,. 
\end{equation*}
Thus, since $\Delta$ is nonnegative self-adjoint and $t\geq 0$, by Theorem \ref{thm:euler-product},
\[\det(t+\Delta)=|\det(t+\Delta)|
\geq (\ul{z}+t)^{\ul{n}} \prod_{\cyc\in \cP(\rG)} \big(1-\ul{p}^{\ul{e}(\cyc)} \big)^{N}=\det(t+\Delta_{0})^{N}\]
and the proof is complete.
\end{proof}

\section{Examples}\label{sec:examples}

\subsection{Acyclic quivers}\label{sec:acyclic-quiver}

As already note in the introduction, if the quiver $\rG$ is acyclic (see Figure~\ref{fig:acyclic-quiver}), then $\det \Delta=0$.

\begin{figure}[!ht]
    \centering
    \includegraphics[scale=.7]{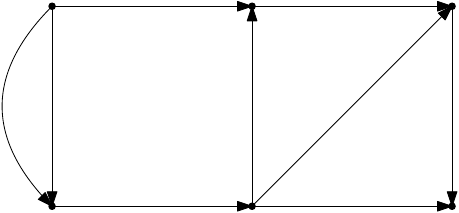}
    \caption{\small An acyclic quiver.}
    \label{fig:acyclic-quiver}
\end{figure}

This is consistent with Theorem \ref{thm:polcarDelta} since the sum in that case consists just of the term $\ul{z}^{\ul{n}}$. However, a finite acyclic quiver necessarily has a vertex with outdegree equal to $0$: otherwise, one could follow the arrows until reaching a vertex already visited, thus creating a cycle. Hence, there is an $a$ such that $z_a=0$ and the determinant is indeed zero.

Since there are no cycles, hence no prime cycles, Theorem \ref{thm:euler-product} reads in that case $\det \Delta= (\ul{z}+\ul{\kappa})^{\ul{n}}$. This is consistent with the direct computation of the determinant (since there are no cycles, only fixed points contribute, so the determinant is just the product of diagonal entries).

\subsection{Unicyclic quivers}\label{sec:unicyclic-quiver}

Suppose the quiver is composed of a unique cycle with some directed trees rooted on it, see Figure \ref{fig:crt-quiver}. This means that each vertex has exactly one outgoing edge, and $z_{s(e)}=x_e$ for all edges $e\in \rE$. Then Theorem \ref{thm:finite-euler} implies
\begin{equation}
    \det \Delta = \left(\prod_{e\in \rE} x_e^{n_{s(e)}}\right) \det(I-\hol(c))\,,
\end{equation}
where $c$ is the unique cycle.
\begin{figure}[!ht]
    \centering
    \includegraphics[scale=.8]{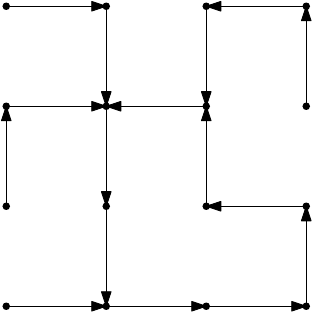}
    \caption{\small A cycle-rooted-tree quiver.}
    \label{fig:crt-quiver}
\end{figure}

\subsection{Quivers with a finite number of prime cycles}\label{sec:example-finite}

If a quiver has only simple cycles, and they are connected in an arborescence way in such a way that there is only one limit cycle, then there are only a finite number of prime cycles, which are exactly those simple cycles. 

Reciprocally, if a quiver has a non simple cycle, then it also has infinitely many prime cycles (if the cycle is written $ab$ with $a$ and $b$ two sub-cycles, then we can create an infinite number of prime words in $a$ and $b$ which correspond thus to prime cycles).

\begin{figure}[!ht]
    \centering
    \includegraphics[scale=.9]{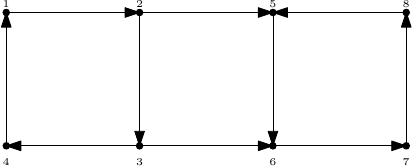}
    \caption{\small Quiver with only two prime cycles: $\cy 1,2,3,4 \yc$ and $\cy 5,6,7,8 \yc$. Each vertex has outdegree $1$ except for vertices $2$ and $3$ which have outdegree $2$.}
    \label{fig:finite-primes}
\end{figure}

In that case, the determinant of the Laplacian associated with any linear representation of that quiver is equal to the finite product given by Theorem \ref{thm:finite-euler}. 
Moreover, note that since the prime cycles are simple and disjoint, we can factor the term $\ul{z}^{\ul{n}}$ inside the determinants and thus get a polynomial factorization of the Laplacian determinant (seen as a polynomial in the~$\ul{x}$ variables).

In the example of Figure \ref{fig:finite-primes}, we have $\cP(\rG)=\{\cyc_1,\cyc_2\}$, where $\cyc_1=\cy 1,2,3,4\yc$ and $\cyc_2=\cy 5,6,7,8\yc$. Moreover, $z_{s(e)}=x_e$ except when $s(e)=2$ or $3$. Thus, $p(c_2)=1$ and $p(c_1)=x_{23}/(x_{23}+x_{25})x_{34}/(x_{34}+x_{36})$. Hence, for any quiver representation:
\begin{align*}
    \det \Delta & = \ul{z}^{\ul{n}} \det\left(I-\frac{x_{23}}{x_{23}+x_{25}}\frac{x_{34}}{x_{34}+x_{36}}\hol(c_1)\right) \det\left(I-\hol(c_2)\right)\,.
\end{align*}

\subsection{Sub-quivers having a finite number of prime cycles}

Introduce two maps $d$ and $\partial$ defined, for $f\in \bigoplus_{v\in \rV} \C_v^{n_v}$ and $e\in \rE$ by 
\begin{equation}
    d f(e)=f(s(e))-U_e f(t(e))\,,
\end{equation}
and for $\omega\in \bigoplus_{e\in \rE} \C_e^{n_{s(e)}}$ and $v\in \rV$ by
\begin{equation}
    \partial \omega(v)=\sum_{\substack{e\in \rE\\ s(e)=v}} \omega(e)\,.
\end{equation}
It then follows from \eqref{eq:covLaplacian} that $\Delta=\partial\circ d$.

Let us denote, for each $a\in\set{p}$ and $I\subset \set{n_a}$, $\C_a^I$ to be the subspace of $\C^{n_a}$ generated by the elements of the canonical basis indexed by $I$.

By the Cauchy--Binet formula, we have
\begin{equation}
    \det \Delta = \sum_{\{I_e:\,e\in \rE\}} \det \left(\partial \circ \proj{\bigoplus_{e\in\rE} \C_{s(e)}^{I_e}} \circ d\right)\,.
\end{equation}
where the sum is over all collections of subsets $\{I_e \subset \set{n_{s(e)}}: e\in \rE\}$ such that 
\begin{equation}
    \sum_{e\in \rE: s(e)=a}\vert I_e\vert=n_a\,.
\end{equation}
Now, we can apply our results to the determinants in the sum. 
Indeed, each term appears as a the data of a sub-quiver along with a representation. 
The sub-quiver has vertices $\rV$ and edges $\rF\subset \rE$ defined by those $e$ for which $I_e\neq \varnothing$. 
The representation is given by the same integers $(n_a: a\in \set{p})$ and matrices $V_e=P_{I_e} U_e$, where $P_{I_e}$ is the orthogonal projection matrix from~$\C_{s(e)}^{n_{s(e)}}$ to $\C_{s(e)}^{I_e}$.

Suppose that one of those sub-quivers has the property that it has only a finite number of prime cycles. Then we can apply the method of the previous Subsection~\ref{sec:example-finite} to compute it.

There will always be such terms as for instance unicyclic sub-quivers will appear in the sum and there weight can be computed using Subsection \ref{sec:unicyclic-quiver}.


\eject

\section{Concluding remarks and questions}\label{sec:conclusion}

We conclude this paper with a list of questions. 
\begin{enumerate}[1.]
\item In \cite[Theorem 9]{Kenyon}, Kenyon proved a variant of Forman's theorem for the case of $\mathrm{SL}_2(\C)$ representations on a bidirected graph, using Moore's notion of Qdeterminant and seeing the Laplacian as a matrix with noncommutative entries. Using a relation attributed to Dyson--Mehta linking this Qdeterminant with the classical one, Kenyon's expression is 
\begin{equation}\label{eq:sl2}
    \det \Delta = (\mathrm{Qdet} \Delta)^2 =\bigg(\sum_{\xi\in \Xi} \ul{x}^\xi \prod_{\cyc\in C(\xi)} (2-\Tr(\hol(\cyc))\bigg)^2\,.
\end{equation}
Expanding this square yields an expression which resembles that of Theorem \ref{thm:formanN}.
It seems likely that one could recover \eqref{eq:sl2} from Theorem \ref{thm:formanN}, using for instance the relations $\Tr(AB)+\Tr(AB^{-1})=\Tr(A)\Tr(B)$ to remove crossings in cycles.
\item 
Can one compute more precisely the term in the right-hand side of \eqref{eq:vector-fields}, for instance using the theory of heaps of pieces \cite{Viennot}? This appears to be closely related to the reinterpretations given in Subsection \ref{sec:reinterpretations}.
\item Can one pass to the large $N$ limit using Voiculescu's theorem (see \cite[Section 5]{Anderson-Guionnet-Zeitouni}) for limits of normalized traces of words in random matrices? Given a fixed quiver, is there a limit identity for the large $N$ limit?
\item 
If $\rG$ is embedded in a manifold $\Sigma$ and if $\ul{U}$ represents a flat connection on $\Sigma$, then the quantities $\Tr\hol(c)$ are constant for $c$ varying in a homotopy class of loops. Thus, $\det \Delta$ may be seen as a partition function for configurations of homotopy classes of multi-loops on $\Sigma$. In the case of a surface $\Sigma$, the functions $\prod_{c\in \Cyc}\Tr\hol(c)$ form a basis of the space of functions on the character variety of $\Sigma$, when $\Cyc$ runs in the set of multisets of simple curves on $\Sigma$ (laminations), by a result of Fock and Goncharov \cite{Fock-Goncharov}. We may wonder if our identities can help to extract the coefficients of $\det \Delta$ in this basis. This is similar in spirit to questions addressed in \cite{Kenyon-dd, Kassel-Kenyon} in rank $1$ and in \cite{Kenyon-webs} for general rank and in particular in the case of ${\rm SU}(3)$ using a new type of skein relations~\cite{Frohman-Sikora}.
\item In the case where $\rG$ is bidirected, all integers $n_a$ are equal to a common $N$, and the matrices~$U_e$ satisfy $U_{e^{-1}}=U_e^{-1}$ and are unitary, the Laplacian is a self-adjoint nonnegative operator. In that case, the coefficients of the monomials $\ul{x}^J$ in $\det \Delta$ are nonnegative (see for instance~\cite[Appendix A]{Cimasoni-Kassel}. Can we find an expression for $\det \Delta$ where the coefficients are manifestly positive (as is the case for Forman or Kenyon's theorems)? When we are in the above setup on a surface with unitary connection, do these positive coefficients give rise to interesting probability measures on homotopy classes of multicurves on surfaces? Can they be studied using the techniques mentioned in the previous item? 
\item 
When all $n_a=N$ and all $U_e=I_N$, we know that the constant vectors are in the kernel of~$\Delta$, hence its determinant vanishes. Can we exploit this observation to simplify the right-hand side of \eqref{eq:vector-fields}? What do the polynomial identities in $\ul{x}$ they correspond to mean?
More generally, for a quiver representation with nontrivial kernel, $\det \Delta=0$ and so we get in that case polynomial identities expressing the nullity of the coefficients of $\det \Delta$ seen as a polynomial in $\ul{x}$. Are these meaningful polynomial identities on the combinatorics of graphs, or just easy cancellations?
\item Can one adapt our results to the case where there are self-loops? This is used for instance in~\cite{KL1} and using the shearing map there would allow to treat Laplacians with potentials.
\item In \cite{KL2} we computed identities for the $\tau$-determinant. We have checked that our theorems above do not hold in general for the $\tau$-determinants. However, can we use the results from~\cite{KL2}, namely Theorem~\ref{thm:appendix}, to give another proof of Theorem~\ref{thm:main} starting from Corollary~\ref{coro:main}?
\item In the context of the dimer model, since Laplacians are roughly obtained as squares of Kasteleyn matrices (even in the case of a weighting by a local system, see e.g. \cite{Dubedat-CR}), is there a link between our expressions for the Laplacian determinant and the formulas obtained in~\cite{Kenyon-webs} for the Kasteleyn determinants in terms of webs and multitraces?
\item Can one establish a link between the Ihara zeta function and the discrete analogue of \eqref{eq:ray-singer}?
\end{enumerate}

\newpage


\bibliographystyle{hmralphaabbrv}
\bibliography{trace-identities}


\appendix

\section{Erratum to \cite{KL2} when $\tau(1)\ne 1$}\label{sec:erratum}

In \cite[Eq. (1)]{KL2}, we defined a notion of $\tau$-determinant for a general tracial map $\tau$. However, in our combinatorial expansion results, we implicitly made the assumption that $\tau(1)=1$. Let us reformulate here our results in the case where this equality is not assumed. For the sake of easier comparison, we formulate the results using the notation of \cite{KL2}.

\begin{theorem}[Corrects Theorem 3.1 of \cite{KL2}]\label{thm:appendix}
In the ring $S=K[a_{ij}: (i,j)\in\rE]$, 
\begin{equation*}
{\det}_{\tau} \Delta_{[m]}=\tau(1)^m \sum_{\rF\in\scF_m} a_{\rF} \prod_{c\in\mathscr{C}(\rF)} \left(1-\tau(1)^{-\ell(c)}\tau(h_c)\right)\,.
\end{equation*}
\end{theorem}

\begin{corollary}[Corrects Corollary 3.2 of \cite{KL2}]
In the quotient ring $S/(a_{ij}-a_{ji}: (i,j)\in\rE]$, 
\begin{equation*}
{\det}_{\tau} \Delta_{[m]}=\tau(1)^m \sum_{[\rF]\in\mathscr{F}_m} a_{\rF} \prod_{\substack{c\in\mathscr{C}(\rF)\\ \ell(c)=2}} \left(1-\tau(1)^{-\ell(c)}\tau(h_c)\right)\prod_{\substack{c\in\mathscr{C}(\rF)\\ \ell(c)\ge 3}} \left(2-\tau(1)^{-\ell(c)}\left(\tau(h_c)+\tau(h_{c^{-1}})\right)\right)\,.
\end{equation*}
\end{corollary}

\begin{theorem}[Corrects Theorem 5.1 of \cite{KL2}]
In the ring $S=K[a_{ij},(i,j)\in \E]$, 
\[{\det}_{\tau}  \Delta_{[Nm]}=\tau(1)^{Nm}\sum_{\rF\in \HF_{m,N}} a_{\rF} \prod_{c\in \Hcyc(\rF)} \big(1-\tau(1)^{-\ell(c)}\tau(\hb_{c})\big) \prod_{c\in \mathscr{S}(\rF)}\big(-\tau(1)^{-\ell(c)}\tau(\hb_{c})\big)\,.\]
\end{theorem}

\begin{corollary}[Corrects Corollary 5.2 of \cite{KL2}]
In the quotient ring $S/(a_{ij}-a_{ji}:(i,j)\in \E)$,
\begin{align*}
{\det}_{\tau}  \Delta_{[Nm]}=& \tau(1)^{Nm}\sum_{[\rF]\in \mathscr{H}\!\mathscr{F}_{m,N}} a_{\rF} \prod_{\substack{c\in \Hcyc(\rF)\\ \ell(c)=2}} \big(1- \tau(1)^{-\ell(c)}\tau(h_{c})\big) \\
& \hspace{1cm} \prod_{\substack{c\in \Hcyc(\rF)\\ \ell(c)\geq 3}} \big(2- \tau(1)^{-\ell(c)}\left(\tau(h_{c})+\tau(h_{c^{-1}}\right))\big)\\
&\hspace{2cm}\prod_{\substack{c\in \mathscr{S}(\rF)\\ \ell(c)=2}} (-\tau(1))^{-\ell(c)}\tau(h_{c}) \prod_{\substack{c\in \mathscr{S}(\rF)\\ \ell(c)\geq 3}} (-\tau(1))^{-\ell(c)}\left(\tau(h_{c})+ \tau(h_{c^{-1}})\right)
\,.
\end{align*}
\end{corollary}

\begin{theorem}[Corrects Theorem 6.1 of \cite{KL2}]
In the ring $S=K[a_{ij}:(i,j)\in \rE]$,
\begin{equation*}
    {\det}_{\tau} \Delta_{[m]}=\left[\frac{\tau(1)}{d}\right]^m \sum_{\rF\in\scF_m} a_{\rF} \prod_{c\in\mathscr{C}(\rF)} \left(1-\left[\frac{d}{\tau(1)}\right]^{\ell(c)}\varepsilon_c \tau(h_c)\right)\,.
\end{equation*}
\end{theorem}

\begin{theorem}[Corrects Theorem 7.1 of \cite{KL2}]
In the ring $S=K[a_{ij}: (i,j)\in\rE]$, 
\begin{equation*}
{\det}_{\tau} \Delta_{[Nm]}=\bigg[\frac{\tau(1)}{N}\bigg]^{mN} \sum_{\rF\in\scF_{m,N}} a_{\rF} \prod_{c\in\mathscr{C}(\rF)} \left(1-\bigg[\frac{N}{\tau(1)}\bigg]^{\ell(c)}\tau(\hb_c)\right)\,.
\end{equation*}
\end{theorem}

In particular, when $\tau(1)=N$ (which is the case when taking $H=M_N(\C)$, $K=\C$, and $\tau:H\to K$ given by the trace $\tau(\cdot)=\Tr(\cdot)$) the last equation simplifies nicely to give the following corollary which did not appear in \cite{KL2}.

\begin{corollary}
Assume that $\tau(1)=N$. In the ring $S=K[a_{ij}: (i,j)\in\rE]$, 
\begin{equation*}
{\det}_{\tau} \Delta_{[Nm]}= \sum_{\rF\in\scF_{m,N}} a_{\rF} \prod_{c\in\mathscr{C}(\rF)} \left(1-\tau(\hb_c)\right)\,.
\end{equation*}
\end{corollary}

\vfill

\end{document}